\documentclass[12pt,reqno]{amsart}
\usepackage{amsmath,amssymb,amsthm,mathtools,calc,verbatim,enumitem,tikz,url,hyperref,mathrsfs,cite,fullpage}
\usepackage{bbm}
\usepackage{stmaryrd}
\usepackage{textcomp}
\usepackage{setspace}
\usepackage{amssymb}
\usepackage{amsthm}
\usepackage{amsmath}
\usepackage{graphicx}
\usepackage{marvosym}
\usepackage{empheq}
\usepackage{latexsym}
\usepackage{fontenc}
\usepackage{color}

\newtheorem{theorem}{Theorem}[section]
\newtheorem{lemma}[theorem]{Lemma}
\newtheorem{corollary}[theorem]{Corollary}

\newtheorem{observation}[theorem]{Observation}
\newtheorem{prop}[theorem]{Proposition}

\newtheorem*{claim*}{Claim}

\theoremstyle{definition}
\newtheorem{definition}[theorem]{Definition}

\newtheorem*{qu*}{Question}
\theoremstyle{remark}

\newcommand\N{\mathbb{N}}
\newcommand\R{\mathbb{R}}
\newcommand\Z{\mathbb{Z}}

\newcommand\cA{\mathcal{A}}

\newcommand\cD{\mathcal{D}}
\newcommand\cG{\mathcal{G}}
\newcommand\cH{\mathcal{H}}

\newcommand\cQ{\mathcal{Q}}

\newcommand\cT{\mathbb{T}}

\newcommand\lcm{\operatorname{lcm}}
\newcommand\eps{\varepsilon}
\renewcommand\leq{\leqslant}
\renewcommand\geq{\geqslant}
\renewcommand\le{\leqslant}
\renewcommand\ge{\geqslant}
\renewcommand\to{\rightarrow}

	\def\eps{\varepsilon}

	\def\R{\mathbb{R}}
	\def\Z{\mathbb{Z}}
	
	\def\Prob{\mathbb{P}}	
	\def\l{\lambda}

	\def\N{\mathbb{N}}
	
	\def\P{\mathcal{P}}
	\def\AA{\mathcal{A}}
	\def\cA{\mathcal{A}}
	
	\def\cQ{\mathcal{Q}}	
	
	\def\FF{\mathcal{F}}

	\def\g{\gamma}
	\def\G{\Gamma}

	\def\p{\partial }	
	\def\la{\langle }
	\def\ra{\rangle }
	\def\<{\langle }
	\def\>{\rangle }

\begin{document}

\title{The structure and number of Erd\H{o}s covering systems}
\author{Paul Balister \and B\'ela Bollob\'as \and Robert Morris \and \\ Julian Sahasrabudhe \and Marius Tiba}

\address{Mathematical Institute, University of Oxford, Radcliffe Observatory Quarter, Woodstock Road, Oxford, OX2 6GG, UK}\email{Paul.Balister@maths.ox.ac.uk}

\address{Department of Pure Mathematics and Mathematical Statistics, Wilberforce Road,
Cambridge, CB3 0WA, UK, and Department of Mathematical Sciences,
University of Memphis, Memphis, TN 38152, USA}\email{b.bollobas@dpmms.cam.ac.uk}

\address{IMPA, Estrada Dona Castorina 110, Jardim Bot\^anico,
Rio de Janeiro, 22460-320, Brazil}\email{rob@impa.br}

\address{Department of Pure Mathematics and Mathematical Statistics, Wilberforce Road, Cambridge, CB3 0WA, UK}
\email{jdrs2@cam.ac.uk}

\address{IMPA, Estrada Dona Castorina 110, Jardim Bot\^anico,
Rio de Janeiro, 22460-320, Brazil}\email{mt576@dpmms.cam.ac.uk}

\thanks{The first two authors were partially supported by NSF grant DMS 1600742, the third author was partially supported by FAPERJ (Proc.~E-26/202.993/2017) and CNPq (Proc.~304237/2016-7), and the fifth author was supported by a Trinity Hall Research Studentship.}

\begin{abstract}
Introduced by Erd\H{o}s in 1950, a \emph{covering system} of the integers is a finite collection of arithmetic progressions whose union is the set $\Z$. Many beautiful questions and conjectures about covering systems have been posed over the past several decades, but until recently little was known about their properties. Most famously, the so-called minimum modulus problem of Erd\H{o}s was resolved in 2015 by Hough, who proved that in every covering system with distinct moduli, the minimum modulus is at most $10^{16}$. 

In this paper we answer another question of Erd\H{o}s, asked in 1952, on the \emph{number} of minimal covering systems. More precisely, we show that the number of minimal covering systems with exactly $n$ elements is 
\[ \exp\left( \left(\frac{4\sqrt{\tau}}{3} + o(1)\right) \frac{n^{3/2}}{(\log n)^{1/2}} \right) \] 
as $n \to \infty$, where 
\[ \tau = \sum_{t = 1}^\infty \left( \log \frac{t+1}{t} \right)^2. \] 
En route to this counting result, we obtain a structural description of all covering systems that are close to optimal in an appropriate sense. 
\end{abstract}
	
	\maketitle 
\section{Introduction}

\enlargethispage*{\baselineskip}
\thispagestyle{empty}

A \emph{covering system\/} is a finite collection of arithmetic progressions that covers\footnote{We emphasize that we do \emph{not} require the progressions to be disjoint. For related work on covering systems with this additional property (sometimes called \emph{exactly covering systems}), see for example~\cite{F,GGRS,M,Z}.} the integers. Erd\H{o}s~\cite{E50} initiated the study of covering systems in 1950, and since then numerous beautiful questions have been asked about their properties (see, for example,~\cite{E50,E52,E73a,E73b,E77,E80,E95,EG,FFKPY,PS,Sch}). Until recently little progress had been made on these problems, but following groundbreaking work of Filaseta, Ford, Konyagin, Pomerance and Yu~\cite{FFKPY} in 2007, a fundamental result was obtained by Hough~\cite{H}, who resolved a problem from the original paper of Erd\H{o}s~\cite{E50} by proving that there do not exist covering systems with distinct moduli and arbitrarily large minimum modulus. Building on his work, the authors of this paper~\cite{BBMST1,BBMST2} recently made further progress on several related open problems. 

In this paper we will study another problem on covering systems, whose study was initiated by Erd\H{o}s~\cite{E52} in 1952: 
$$\textup{How many minimal covering systems of size~$n$ are there?}$$ 
Erd\H{o}s~\cite{E52} gave a simple proof that there are only finitely many minimal\footnote{A covering system $\cA$ is \emph{minimal} if no proper subset of it covers $\Z$. Without this restriction there are infinitely many covering systems of size $2$, since we can take $\cA = \{ \Z, A \}$ for any arithmetic progression $A$.} covering systems of size~$n$, but the bound he obtained on their number was doubly exponential. A more reasonable upper bound follows from a result of Simpson~\cite{Simp}, who proved in 1985 (see Section~\ref{structure:sec}) that the largest modulus in a minimal covering system of size $n$ is at most $2^{n-1}$. Note that this bound is best possible, since $\cA = \big\{ 2^{i-1} \pmod {2^i} : i \in [n-1] \big\} \cup \big\{ 0 \pmod {2^{n-1}} \big\}$ is a minimal covering system, and that it easily implies that there are at most $2^{O(n^2)}$ minimal covering systems of size $n$. We will show that there are in fact rather fewer such systems, and we will moreover determine asymptotically the logarithm of their number. The main aim of this paper is to prove the following theorem.

\begin{theorem}\label{thm:MainCountingThm} 
The number of minimal covering systems of\/ $\Z$ of size $n$ is 
\begin{equation}\label{eq:main:count}
\exp\left( \left( \frac{4\sqrt{\tau}}{3} + o(1)\right)\frac{n^{3/2}}{(\log n)^{1/2}  } \right)
\end{equation}
as $n \to \infty$, where 
\[ 
\tau =  \sum_{t = 1}^{\infty} \left(\log \frac{t+1}{t} \right)^2. 
\]
\end{theorem} 

We remark that proving a weaker upper bound, with a different constant in the exponent, is significantly easier, and we will give a short proof of such a bound in Section~\ref{counting:weird:frames:sec}. Let us also note here that we will prove the lower bound under the additional restriction that the moduli are distinct, and so the conclusion of Theorem~\ref{thm:MainCountingThm} also holds for such systems. 

In order to motivate the form of the formula~\eqref{eq:main:count}, let us begin by describing a simple construction that gives a slightly weaker lower bound. Let $p_1 < \ldots < p_k$ be the first $k$ primes, and for each $i \in [k]$, choose $p_i - 1$ arithmetic progressions $A^{(i)}_1,\ldots,A^{(i)}_{p_i-1}$ with the following properties: for each $j \in [p_i-1]$, the modulus of $A^{(i)}_j$ is divisible by $p_i$ and divides $Q_i := p_1\cdots p_i$, and $A^{(i)}_j$ contains $j \cdot Q_{i-1}$. It is not difficult to show that, for each such choice, by adding the progression $\big\{ 0 \pmod {Q_k} \big\}$ we obtain a distinct minimal covering system of size $n = \sum_{i = 1}^k (p_i - 1) + 1 \approx k^2 \log k$. Since we have $2^{i-1}$ choices for the progression $A^{(i)}_j$ for each $i \in [k]$ and $j \in [p_i-1]$, this implies that there are at least 
$$\prod_{i = 1}^k 2^{(i-1)(p_i - 1)}\, = \, \exp\Big( \Omega\big( k^3 \log k \big) \Big) \, = \, \exp\bigg( \frac{\Omega\big( n^{3/2} \big)}{(\log n)^{1/2}  } \bigg)$$
minimal covering systems of $\Z$ of size $n$. In Section~\ref{arithmetic:sec} we will describe a somewhat more complicated construction that proves the lower bound in Theorem~\ref{thm:MainCountingThm}. 

We will refer to collections of progressions as in the construction above as ``frames" (see Section~\ref{structure:sec} for a precise definition). The second main result of this paper, and the key step in the proof of Theorem~\ref{thm:MainCountingThm}, will be a structural description of all ``efficient" covering systems; roughly speaking, we will show that every such covering system contains a large ``approximate frame". The purpose of the next section is to state this structural theorem.

\section{The structure of efficient coverings}\label{structure:sec}

{\setstretch{1.12}
In this section we will state our main structural theorem. In order to do so, it will be convenient to shift our attention to the following (slightly more general) \emph{geometric} setting. Let $S_1,\ldots,S_k$ be finite sets with at least two elements and set $S_I := \prod_{i \in I} S_i$ for each $I \subseteq [k]$. If $H = H_1 \times\cdots \times H_k \subseteq S_{[k]}$ with each $H_i$ either equal to $S_i$ or a singleton element of $S_i$, then we say that $H$ is a {\em hyperplane}. We write $F(H) := \{ i \in [k] : |H_i| = 1 \}$ for the \emph{fixed coordinates} of $H$, and $F(\cA) := \bigcup_{H \in \cA} F(H)$ if $\cA$ is a collection of hyperplanes. We will also write $H = [x_1,\ldots,x_k]$, where $x_i \in S_i \cup \{ * \}$ for each $i \in [k]$, and $*$ indicates that $H_i = S_i$. 

\begin{definition}\label{def:frame}
A \emph{simple frame} centred at an element $(s_1,\ldots,s_k) \in S_{[k]}$ (which we call the \emph{axis}), is a sequence $(\FF_1,\ldots,\FF_k)$, where $\FF_i$ is a collection of $|S_i| - 1$ hyperplanes of the form
\begin{equation}\label{eq:frame:elements}
\big[ x_1,\ldots,x_{i-1},a,*,\cdots,*\big],
\end{equation}
one for each $a \in S_i \setminus \{s_i\}$, with $x_j \in \{ s_j, * \}$ for each $j \in [i-1]$.

A \emph{frame} is obtained from a simple frame by permuting the order of the sets $S_1,\ldots,S_k$.
\end{definition}

Observe that if $(\FF_1,\ldots,\FF_k)$ is a frame centred at $(s_1,\ldots,s_k)$, then the collection 
$$\cA := \FF_1 \cup \cdots \cup \FF_k \cup \big\{ [s_1,\ldots,s_k] \big\}$$ 
is a minimal cover of $S_{[k]}$. Indeed, if we remove the hyperplane $\big[ x_1,\ldots,x_{i-1},a,*,\cdots,*\big]$ from $\cA$, then the element $(s_1,\ldots,s_{i-1},a,s_{i+1},\ldots,s_k)$ will be uncovered by the remaining hyperplanes. Note that if we set $S_i = \{0,\ldots,p_i-1\}$ for each $i \in [k]$, 
then the construction given in the introduction is equivalent to a frame centred at $(0,\ldots,0)$. When we (for now informally, but later on precisely) discuss frames in $\Z$, we will always mean that each set $S_i = \{0,\ldots,p-1\}$ for some prime $p$ (these primes will not generally be distinct), and we will map $S_1 \times \cdots \times S_k$ into $\Z_N$, where $N = \prod_{i = 1}^k |S_i|$, using the Chinese Remainder Theorem to identify $\Z_N$ with the product of groups $\Z_{p^\gamma}$, and then expanding base $p$.\footnote{For example, $(a_0,a_1,\dots,a_{\gamma-1}) \in \Z_p \times \dots \times \Z_p$ corresponds to the element $\sum_{i = 0}^{\gamma-1} a_i p^i \in \Z_{p^\gamma}$.} Note that every arithmetic progression in $\Z$ corresponds to a hyperplane, but not every hyperplane corresponds to an arithmetic progression if primes are repeated (see Sections~\ref{arithmetic:sec} and~\ref{MainCountingProofSec}). 
  
The key idea behind the proof of Theorem~\ref{thm:MainCountingThm} is the following (imprecise) conjecture:
$$\textup{``Almost every minimal covering system of $\Z$ of size $n$ is close to a frame."}$$
We will not prove a result of this form; instead, we will use a slightly weaker notion, which we call a \emph{$\delta$-generalized frame}. These objects differ from frames in two key ways: the fixed elements ``to the left" of $i$ in a hyperplane $H \in \FF_i$ are allowed to vary with $i$, and instead of insisting that ``all coordinates to the right are free" (as in~\eqref{eq:frame:elements}), we allow a few ``small" coordinates to be fixed (with the product of their sizes bounded by $1/\delta$). 

The next definition is both important and somewhat technical, and we will need some additional notation. Given a hyperplane $H$, we write $H_i$ for its $i$th coordinate, and for any $I \subseteq [k]$ we will write $H_I = \prod_{i \in I} H_i$ for the hyperplane in $S_I$ obtained by restricting $H$ to the coordinates of $I$, and define $\mu_I(H) := |H_I| \cdot |S_I|^{-1}$ when $I \ne \emptyset$, and $\mu_\emptyset(H) := 1$.}

\begin{definition}[$\delta$-generalized frames]\label{def:generalizedframe}
Let $\delta > 0$, and let $S_1,\ldots,S_k$ be finite sets with at least two elements. A~\emph{simple $\delta$-generalized frame} in $S_{[k]}$ is a sequence $(\FF_1,\ldots,\FF_k)$, where $\FF_i$ is a collection of at most $|S_i| - 1$ hyperplanes, satisfying the following conditions. For each $i \in [k]$, there exists a set $I(i) \supseteq \{ i+1,\ldots,k \}$, and  for each $j \not\in I(i) \cup \{i\}$, there exists an element $s_j(i) \in S_j$, such that, for each $H \in \FF_i$,
$$i \in F(H), \qquad\quad \mu_{I(i)}(H) > \delta \qquad\quad \text{and} \qquad\quad H_j \in \big\{ s_j(i), S_j \big\}.$$
Moreover, if $\min\big\{ |S_i|, |S_j| \big\} \ge \delta^{-1}$ and $i \ne j$, then $\FF_i$ and $\FF_j$ are disjoint. A \emph{$\delta$-generalized frame} is obtained from a simple $\delta$-generalized frame by permuting the sets $S_1,\ldots,S_k$.
\end{definition}

We are now ready to state our main structural theorem for covering systems that contain roughly (up to a constant factor) the same number of elements as a frame. 

\begin{theorem} \label{thm:mainStructureThm} 
For every $C, \eps > 0$ there exists $\delta = \delta(C,\eps) > 0$ so that for every collection of finite sets $S_1,\ldots,S_k$ with at least two elements, the following holds. If\/ $\cA$ is a minimal cover of $S_{[k]}$ with hyperplanes such that $F(\cA) = [k]$ and
\begin{equation}\label{eq:structure:assumption}
|\cA | \leq C \sum_{i=1}^k \big( |S_i| - 1 \big),
\end{equation}
then $\cA$ contains a $\delta$-generalized frame $(\FF_1,\ldots,\FF_k)$, with 
\begin{equation}\label{eq:structure:conclusion}
\sum_{i=1}^k |\FF_i| \geq (1 - \eps) \sum_{i=1}^k \big( |S_i| - 1 \big).
\end{equation}
 \end{theorem}

The theorem above can be thought of as an inverse theorem for the following extremal result of Simpson~\cite{Simp}. If $\cA$ is a collection of arithmetic progressions, then we write $\lcm(\cA)$ for the least common multiple of the moduli of the progressions in $\cA$. 


\begin{theorem}[Simpson's theorem]\label{Simpson:thm}
If $\cA$ is a minimal cover of $S_{[k]}$ with hyperplanes such that $F(\cA) = [k]$, then
$$|\cA | \ge \sum_{i=1}^k \big( |S_i| - 1 \big) + 1.$$
In particular, if $\cA$ is a minimal covering system of $\Z$ with $\lcm(\cA) = p_1^{\g_1} \cdots p_m^{\g_m}$, then
$$|\cA | \ge \sum_{i=1}^m \g_i \big( p_i - 1 \big) + 1.$$
\end{theorem}

In the appendix, we will provide (for the reader's convenience) a proof of Simpson's theorem. Let us also remark here that, while the form of the function $\delta(C,\eps)$ will not matter for our purposes, we will prove that Theorem~\ref{thm:mainStructureThm} holds with $\delta = (\eps / C)^{O(\log(1 / \eps))}$. 

In order to deduce Theorem~\ref{thm:MainCountingThm} from Theorem~\ref{thm:mainStructureThm}, we will need to count $\delta$-generalized frames quite precisely, and show that there are relatively few choices for the remaining elements; we will also need to show that there are few minimal covering systems that fail to satisfy~\eqref{eq:structure:assumption}. These calculations are carried out in Sections~\ref{counting:weird:frames:sec} and~\ref{MainCountingProofSec}.

\subsection{An outline of the proof of the structural theorem}

\enlargethispage*{\baselineskip}
{\setstretch{1.07}

The proof of Theorem~\ref{thm:mainStructureThm} requires a few somewhat technical definitions, and to prepare the reader for these we will begin by giving an outline of the argument. The idea is to construct a tree that encodes the structure of the covering system by `exploring' it coordinate by coordinate. To be more precise, given a minimal cover $\cA$ of $S_{[k]}$, let us choose a coordinate $i \in [k]$ to explore, and observe (see Section~\ref{operation:sec} for the details) that for each $s \in S_i$ we obtain a covering system of 
$$S_1 \times \cdots \times S_{i-1} \times \{s\} \times S_{i+1} \times \cdots \times S_k,$$
which we identify with $S_{[k] \setminus \{i\}}$. (Here the hyperplanes $H \in \cA$ with $H_i = S_i$ appear in each of the $|S_i|$ covering systems corresponding to coordinate~$i$.) These covering systems may not be minimal, but for each $s \in S_i$ we can take a minimal sub-covering~$\cA_s$. 

Now, some of the systems $\cA_s$ may be trivial (i.e., may consist of a single hyperplane), and when this occurs we are happy, because such hyperplanes can be used in the frame that we are trying to construct. For the remaining elements $s \in S_i$, we consider the set of fixed coordinates $F(\cA_s)$ of $\cA_s$, and observe (see Lemma~\ref{lem:coveringfixedpoints}) that every coordinate (except~$i$) is in $F(\cA_s)$ for some $s \in S_i$. We may now choose, for each $s$ such that $F(\cA_s)$ is non-empty, a coordinate $j \in F(\cA_s)$, and repeat the above construction, exploring the minimal covering system $\cA_s$, starting with the coordinate~$j$. Iterating this process produces a rooted tree (which we call an `index tree', see Definition~\ref{def:index:tree}), each of whose vertices is labelled with a set $I \subseteq [k]$ and a coordinate $i \in I$, which are the fixed coordinates of the corresponding minimal covering system, and the coordinate `explored' at that vertex, respectively. 

So far, we have not said anything about how to construct the sets $\FF_i$, or how to choose the coordinate~$i$ that we explore in a given step. For simplicity, let us explain this only for the first step (the choice for later steps is similar). First, if there exists $i \in [k]$ such that there are at least $(1-\eps)\big( |S_i| - 1 \big)$ hyperplanes $H \in \cA$ with $i \in F(H)$ and $\mu_{[k] \setminus \{i\}}(H) > \delta$, then we choose such a coordinate~$i$ to explore, and associate this collection of `frame-like' hyperplanes with the current vertex (in this case, the root of the tree). One of the key ideas of the paper is that, if such a collection of hyperplanes does not exist for any $i \in [k]$, then we may use the Lov\'asz Local Lemma to deduce (see Lemma~\ref{lem:localLem}) that there exists a coordinate~$j$ (which we will choose to explore), and a `large' collection $\cG$ of hyperplanes in the current collection, such that~$j$ is a fixed coordinate of each. This collection of `garbage' hyperplanes will later be used, together with~\eqref{eq:structure:assumption}, to show that this case does not occur too often. 

The plan described above is carried out in Section~\ref{exploring:sec}, the main result being Lemma~\ref{lem:explorationTree}, which states that if $\delta$ is sufficiently small, then there exists a suitable `exploration tree' $\cT$ of $\cA$ (see Definition~\ref{def:exploration:tree}). This exploration tree can be very large, however, and to extract our $\delta$-generalized frame from it we will need to choose a suitable sub-tree $T$. To do so, we choose $k$ `special' vertices of $\cT$, one for each coordinate, and take the union of the paths from these vertices to the root. If almost all of these special vertices are `good' (that is, we found a large collection of frame-like hyperplanes when exploring them), then we obtain a sufficiently large $\delta$-generalized frame. On the other hand, if a positive proportion of them are `bad', then we use the `garbage' hyperplanes to show that inequality~\eqref{eq:structure:assumption} cannot hold. In order to carry out this argument, we need to choose the special vertices carefully; it turns out that it is sufficient to choose them via a depth-first search, see Section~\ref{frame:structure:sec} for the details.}

\section{Exploring the cover}\label{exploring:sec}

{\setstretch{1.12}

In this section we will take the first step towards Theorem~\ref{thm:mainStructureThm} by describing a much larger object that is somewhat easier to construct, the \emph{exploration tree}. To define these, we first need to introduce the following simpler objects, which we call \emph{index trees}. Let us fix, for the rest of the proof of Theorem~\ref{thm:mainStructureThm}, a collection of finite sets $S_1,\ldots,S_k$ with at least two elements, and let us write $N(u)$ for the set of out-neighbours of $u$ in a rooted tree, where we orient the edges away from the root. 

\begin{definition}\label{def:index:tree}
An \emph{index tree} $\cT$ of $[k]$ is a rooted tree, equipped with a labelling of its vertices $u \mapsto (I_u,i_u)$, where $I_u \subseteq [k]$ and $i_u \in I_u$, that satisfies the following conditions:
\begin{itemize}
\item[$(i)$] the root of $\cT$ has label $([k],i)$ for some $i \in [k]$;
\item[$(ii)$] for each vertex $v \in V(\cT)$,  
\begin{equation}\label{eq:index:tree}
\bigcup_{u \in N(v)} I_u = I_v \setminus \{ i_v \}.
\end{equation}
\end{itemize}
\end{definition}

We can now define the exploration tree of a collection of hyperplanes in $S_{[k]}$. Given a rooted tree $T$ and vertices $u,v \in V(T)$, let us write $u \prec_T v$ to indicate that $u$ lies on the path from $v$ to the root (so, in particular, $v \prec_T v$).

\begin{definition}\label{def:exploration:tree}
Let $\lambda,\eps, \delta >0$, and let $\cA$ be a collection of hyperplanes in $S_{[k]}$. An $(\lambda,\eps,\delta)$-\emph{exploration tree} of $\cA$ is an index tree $\cT$ of $[k]$ such that, for each vertex $u \in V(\cT)$: 
\begin{itemize}
\item[$(a)$]\label{always:condition} 
there exists a collection $\cA_u \subseteq \cA$ such that $\cA_u':= \big\{ H_{I_u} : H \in \cA_u \big\}$ is a minimal cover of $S_{I_u}$ with $F(\cA'_u) = I_u$, and if $u \in N(v)$, then:\smallskip 
\begin{itemize}
\item[$(i)$] $\cA_u \subseteq \cA_v$;\smallskip
\item[$(ii)$] $F(\cA_u) \subseteq \big\{ i_w : w \prec_\cT v \big\} \cup I_u$;\smallskip
\item[$(iii)$] there exists an element $s_u \in S_{i_v}$ such that $H_{i_v} \in \{s_u,S_{i_v}\}$ for each $H \in \cA_u$.
\end{itemize}
\end{itemize}
Moreover, for each vertex $u \in V(\cT)$, one of the following holds:
\begin{itemize}
\item[$(b)$] 
$u$ is \emph{good}, which means that there exists a collection of hyperplanes $\FF_u \subseteq \AA_u$, with 
\begin{equation} \label{cond:BigFrame} 
|\FF_u| \ge (1 - \eps) \big( |S_{i_u}| - 1 \big),
\end{equation}
such that $i_u \in F(H)$ and $\mu_{I_u \setminus \{i_u\}}(H) > \delta$ for each $H \in \FF_u$.\smallskip
\item[$(c)$]
$u$ is \emph{bad}, which means that there exists a collection of hyperplanes $\cG_u \subseteq \cA_u$, with 
\begin{equation}\label{cond:BigGarbage} 
\sum_{H \in \cG_u} 2^{- |F(H) \cap I_u|/4} \geq \frac{|S_{i_u}|}{\lambda},
\end{equation}
such that $i_u \in F(H)$ and $|F(H) \cap I_u| \ge 2$ for each $H \in \cG_u$. 
\end{itemize}
\end{definition}

We think of the elements of $\FF_u$ (when $u$ is good) and $\cG_u$ (when $u$ is bad) as hyperplanes that (respectively) do and do not look like parts of a frame from the perspective of the vertex~$u$. We will show (see Lemma~\ref{lem:explorationTree}, below) that exploration trees always exist, as long as we choose $\delta$ to be sufficiently small, depending on $\lambda$ and $\eps$. We will then, in Section~\ref{frame:structure:sec}, carefully choose a subtree $T$ of our exploration tree $\cT$, and one `special' vertex for each coordinate $i \in [k]$, with the following three properties: the frames corresponding to good special vertices are disjoint (unless one of the corresponding sets $S_i$ is very small); if `many' special vertices of $T$ are bad, then~$\cA$ fails to satisfy~\eqref{eq:structure:assumption}; and if `almost all' of the special vertices of $T$ are good, then there exists a sufficiently large $\delta$-generalized frame in $\cA$.

The main aim of this section is to prove the following lemma. 

\begin{lemma}\label{lem:explorationTree} 
Let $\lambda,\eps \in (0,1)$, and let $\cA$ be a minimal cover of $S_{[k]}$ with hyperplanes such that $F(\cA) = [k]$. If 
\begin{equation}\label{delta:bound}
0 < \delta < 2^{-9} \lambda^2 \eps^{2\log_2 (1/\lambda \eps) + 11},
\end{equation}
then there exists an $(\lambda,\eps,\delta)$-exploration tree of $\cA$.
\end{lemma}

Let us fix, for the rest of this section, constants $0 < \lambda, \eps < 1$ and $\delta > 0$ satisfying~\eqref{delta:bound}. We will prove Lemma~\ref{lem:explorationTree} by iteratively extending a `partial' exploration tree by applying the following lemma to a leaf of the current tree. 

\begin{lemma}\label{one:step:exploration}
Let $\emptyset \ne I \subseteq [k]$, and let $\cA$ be a minimal cover of $S_I$ with $F(\cA) = I$. 
\begin{itemize}
\item[$(a)$]
For each $i \in I$, there exists a map $J \colon S_i \to \P(I \setminus \{i\})$, with 
\begin{equation}\label{eq:J:union}
\bigcup_{s \in S_i} J(s) = I \setminus \{i\},
\end{equation}
and for each $s \in S_i$ there exists $\cA_s \subseteq \cA$ such that $\cA_s':= \big\{ H_{J(s)} : H \in \cA_s \big\}$ is a minimal cover of $S_{J(s)}$, $F(\cA_s) \setminus \{i\} = J(s)$, and $H_i \in \{s,S_i\}$ for each $H \in \cA_s$.
\end{itemize}
Moreover, there exists $i \in I$ such that one of the following holds:
\begin{itemize}
\item[$(b)$]
there exists a collection $\FF \subseteq \cA$, with 
$$|\FF| \ge (1 - \eps) \big( |S_i| - 1 \big),$$
such that $i \in F(H)$ and $\mu_{I \setminus \{ i \}}(H) > \delta$ for each $H \in \FF$. \smallskip
\item[$(c)$]
there exists a collection $\cG \subseteq \cA$, with 
\begin{equation}\label{eq:BigGarbage} 
\sum_{H \in \cG} 2^{- |F(H)|/4} \ge \frac{|S_i|}{\lambda},
\end{equation}
such that $i \in F(H)$ and $|F(H)| \ge 2$ for each $H \in \cG$. 
\end{itemize}
\end{lemma}

Let us fix a set $\emptyset \ne I \subseteq [k]$ until the end of the proof of Lemma~\ref{one:step:exploration}. This section is organised as follows: in Section~\ref{tech:lem:sec} we will prove two (straightforward) technical lemmas, in Section~\ref{operation:sec}, we will introduce the operation that we will use to construct the map $J$ and the families $\cA_s$, and in Section~\ref{exploring:proof:sec}, we will prove Lemma~\ref{one:step:exploration}, and deduce Lemma~\ref{lem:explorationTree}.

\subsection{Two technical lemmas}\label{tech:lem:sec}

{\setstretch{1.12}
Our first technical lemma (Lemma~\ref{lem:localLem}, below) follows from a straightforward application of the Lov\'asz Local Lemma. We will apply it, in the case that there does not exist a collection $\FF \subseteq \cA$ as in Lemma~\ref{one:step:exploration}$(b)$ for any $i \in I$, to a certain subset $\prod_{i \in I} R_i \subseteq S_I$, in order to find an index such that~$(c)$ holds. Our second technical lemma (Lemma~\ref{lem:MassOnRestrictedSpace}, below) will allow us to deduce the bound~\eqref{eq:BigGarbage} from the condition given by the local lemma. Let us say that a hyperplane $H$ in $R_I$ is \emph{non-trivial} if $H \ne R_I$. 


\begin{lemma} \label{lem:localLem}
Let $0 < \eta < 1/5$ and let $\{ R_i : i \in I \}$ be a collection of finite sets, each with at least two elements. Let $\cA$ be a collection of non-trivial hyperplanes in $R_I$, and let $\tilde\mu$ denote the uniform measure on $R_I$. If
\begin{equation} \label{equ:noheavydirection} 
\sum_{H \in \cA \,:\, i \in F(H)} e^{\eta |F(H)|} \tilde\mu(H) < \frac{\eta}{2}
\end{equation} 
for every $i \in [k]$, then $R_I$ is not covered by the hyperplanes in $\cA$. 
\end{lemma} 

\begin{proof}
We choose a point $y \in R_I$ uniformly at random and apply the local lemma. For each hyperplane $H \in \AA$ we define $E_H$ to be the (``bad'') event that $y \in H$. Observe that $\Prob(E_H) = \tilde\mu(H)$, and define a dependency graph $G$ on the events $\{E_H\}_{H \in \AA}$ by setting $E_H \sim E_{H'}$ if $F(H) \cap F(H') \ne \emptyset$. Observe that if $F(H) \cap \big( F(H^{(1)}) \cup \cdots \cup F(H^{(t)}) \big) = \emptyset$, then $E_H$ is independent of the collection $\big\{ E_{H^{(1)}}, \ldots, E_{H^{(t)}} \big\}$, so $G$ is a valid dependency graph. 

Next, we define weights 
$$x(H) = e^{\eta |F(H)|} \tilde\mu(H)$$
for each $H \in \cA$. To apply the local lemma we need to show that
$$\Prob(E_H) \leq x(H) \prod_{E_H \sim E_{H'}} \big( 1 - x(H') \big).$$
To do so, we first claim that $1 - x(H) \ge e^{-2x(H)}$ for every $H \in \cA$. This holds because
$$x(H) = e^{\eta |F(H)|} \tilde\mu(H) \le \big( e^{\eta} / 2 \big)^{|F(H)|} \le \big( 1 - e^{-1} \big)^{|F(H)|} \le 1 - e^{-1},$$ 
where the first inequality is $\tilde\mu(H) \leq 2^{-|F(H)|}$, which holds because each set $S_i$ has at least two elements, the second follows since $\eta < 1/5$, and the third since the hyperplanes in $\cA$ are non-trivial, so $|F(H)| \ge 1$. Therefore, for each $H \in \cA$, we have
\begin{align*} 
\prod_{E_H \sim E_{H'}} \big( 1 - x(H') \big) & \, \ge \exp \bigg( - 2 \sum_{E_H \sim E_{H'}} x(H') \bigg) \ge \exp\bigg( -2 \sum_{i \in F(H)} \sum_{H' \in \cA \,:\, i \in F(H') } x(H') \bigg)\\
& \, = \exp\bigg( -2 \sum_{i \in F(H)} \sum_{H' \in \cA \,:\, i \in F(H') } e^{\eta |F(H')|} \tilde\mu(H') \bigg) \geq \exp\big( -\eta |F(H)| \big),
\end{align*}
where the last inequality follows from~\eqref{equ:noheavydirection}. This implies that
$$x(H) \prod_{E_H \sim E_{H'}} \big( 1 - x(H') \big) \ge x(H) e^{-\eta |F(H)|} = \tilde\mu(H) = \Prob(E_H),$$
as required. By the local lemma, it follows that the probability that none of the events $E_H$ holds is non-zero, and hence there exists a point $y \in R_I$ that is not covered by $\cA$. 
\end{proof}
}

The second technical lemma is even more straightforward. Recall that $S_1,\ldots,S_k$ are fixed finite sets with at least two elements, and that the (non-empty) set $I \subseteq [k]$ and positive constants $\lambda$ and $\eps$ were fixed above.}

{\setstretch{1.09}

\begin{lemma} \label{lem:MassOnRestrictedSpace}
For each $j \in I$, let $R_j \subseteq S_j$ be such that $|R_j| \ge \eps \big( |S_j| - 1 \big) + 1$, and let $\tilde{\mu}$ denote the uniform measure on $R_I$. Let $H$ be a hyperplane in $S_I$, and let $i \in F(H)$. If  
$$\mu_{I \setminus \{ i \}}(H) \le 2^{-9} \lambda^2 \eps^{2\log_2 (1/\lambda \eps) + 11},$$
then
\begin{equation}\label{eq:SmallMassOnSubsets} 
\tilde{\mu}\big( H \cap R_I \big) \le \frac{\lambda}{2^{|F(H)|/2 + 4}|S_i|}.
\end{equation}
\end{lemma}

\begin{proof} 
Set $\ell := |F(H)|$ and $\delta_0 := 2^{-9}\lambda^2 \eps^{2\log_2 (1/\lambda \eps) + 11}$, and observe that 
$$\tilde{\mu}\big( H \cap R_I \big) \cdot |R_i| \, \le \, \mu_{I \setminus \{i\}} (H) \prod_{j \in F(H), \, i \ne j} \frac{|S_j|}{|R_j|} \, \le \, \delta_0 \eps^{-(\ell-1)}.$$
Now, note that $|R_j| \ge 2$ for every $j \in I$, and suppose that~\eqref{eq:SmallMassOnSubsets} does not hold. Then 
$$2^{-(\ell-1)} \ge \prod_{j \in F(H), \, j \ne i} \frac{1}{|R_j|} \, = \, \tilde{\mu}\big( H \cap R_I \big) \cdot |R_i| \, \ge \, \frac{\lambda}{2^{\ell/2 + 4}} \cdot \frac{|R_i|}{|S_i|} \, \ge \, \frac{\lambda \eps}{2^{\ell/2 + 4}},$$
and hence $\ell \le 2\log_2(1/\lambda \eps) + 10$. It follows that
$$2^{\ell/2 + 4} \cdot \tilde{\mu}\big( H \cap R_I \big) \le \, \frac{2^{9/2} \delta_0}{|R_i|} \bigg( \frac{\sqrt{2}}{\eps} \bigg)^{\ell-1} \le \, \frac{2^{9/2} \delta_0}{|R_i|} \bigg( \frac{\sqrt{2}}{\eps} \bigg)^{2\log_2(1/\lambda \eps) + 9} = \,  \frac{\lambda \eps}{|R_i|}  \, \le \,  \frac{\lambda}{|S_i|},$$
as required.
\end{proof}

\subsection{An operation on a covering system}\label{operation:sec}

We next introduce a simple operation that, given a minimal cover of $S_I$, produces a map $J$ and a collection $\{ \cA_s : s \in S_i \}$ as required by Lemma~\ref{one:step:exploration}$(a)$. This operation is the basic tool we will use in the construction of our exploration trees. Recall that the (non-empty) set $I \subseteq [k]$ was fixed above, and let $\cA$ be a minimal cover of $S_I$ with $F(\cA) = I$. For each $i \in I$ and $s \in S_i$, set 
$$\cH(i,s) := \big\{ H \in \cA : H_i \in \{ s, S_i \} \big\},$$
and observe that the collection $\cH'(i,s) := \big\{ H_{I \setminus \{i\}} : H \in \cH(i,s) \big\}$ is a cover of $S_{I \setminus \{i\}}$. Note that moreover, since $\cA$ is minimal, there is a bijection between $\cH(i,s)$ and $\cH'(i,s)$. Let $\cA^*_s \subseteq \cH'(i,s)$ be an arbitrary minimal subcover of $S_{I \setminus \{i\}}$, and define 
$$J(s) := F(\cA^*_s) \qquad \text{and} \qquad \cA_s := \big\{ H \in \cH(i,s) : H_{I \setminus \{i\}} \in \cA^*_s \big\}.$$ 
Note that $\cA_s' = \big\{ H_{J(s)} : H \in \cA_s \big\}$ is a minimal cover of $S_{J(s)}$, that $F(\cA_s) \setminus \{i\} = J(s)$, and that $H_i \in \{s,S_i\}$ for each $H \in \cA_s$. 
To verify that $J$ and $\{ \cA_s : s \in S_i \}$ satisfy Lemma~\ref{one:step:exploration}$(a)$, it therefore only remains to check that~\eqref{eq:J:union} holds.  

\begin{lemma}\label{lem:coveringfixedpoints}  
Let $\cA$ be a minimal cover of $S_I$ with hyperplanes. If $F(\cA) = I$, then 
$$\bigcup_{s \in S_i} J(s) = I \setminus \{i\}$$
for each $i \in I$. 
\end{lemma} 

\begin{proof}
We will in fact show that for every $H \in \cA$, there exists $s \in S_i$ with $H_{I \setminus \{i\}} \in \cA^*_s$. Since $J(s) = F(\cA^*_s) \subseteq I \setminus \{i\}$ and $F(\cA) = I$, this will be enough to prove the lemma.

To prove the claim, let $x \in S_I$ be an element that is only covered by $H$ (recall that $\cA$ is minimal),  and set $s := x_i$. We claim that $H_{I \setminus \{i\}} \in \cA^*_s$. Indeed, since $\cA^*_s$ is a cover of $S_{I \setminus \{i\}}$, it must cover the vector $x'$ obtained from $x$ by ignoring the $i$th coordinate, and $H_{I \setminus \{i\}}$ is the only potential element of $\cA^*_s$ that can do so. 
\end{proof}
}

\subsection{Construction of the exploration tree}\label{exploring:proof:sec}

Having completed our preparations, we are now ready to prove Lemma~\ref{one:step:exploration}, and deduce Lemma~\ref{lem:explorationTree}.

\begin{proof}[Proof of Lemma~\ref{one:step:exploration}]
Let $\emptyset \ne I \subseteq [k]$, and let $\cA$ be a minimal cover of $S_I$ such that $F(\cA) = I$. To prove part~$(a)$,  for each $i \in I$ we apply the construction defined in Section~\ref{operation:sec} to obtain a map $J \colon S_i \to \P(I \setminus \{i\})$ and a collection $\{ \cA_s : s \in S_i \}$, where $\cA_s \subseteq \cA$, as required. In particular, $\cA_s' = \big\{ H_{J(s)} : H \in \cA_s \big\}$ is a minimal cover of $S_{J(s)}$, $F(\cA_s) \setminus \{i\} = J(s)$, $H_i \in \{s,S_i\}$ for each $H \in \cA_s$, and the map $J$ satisfies~\eqref{eq:J:union} by Lemma~\ref{lem:coveringfixedpoints}.

To prove that there exists $i \in I$ such that either~$(b)$ or~$(c)$ holds, let us define an element $s \in S_i$ to be \emph{special for $i$} if there exists a hyperplane $H \in \cA$ such that 
$$H_i = s \qquad \text{and} \qquad \mu_{I \setminus \{i\}}(H) > \delta.$$ 
If this holds, then we say that the hyperplane $H$ is a \emph{witness} for the pair $(s,i)$. Now, for each $i \in I$ define $S_i^*$ to be the set of elements $s \in S_i$ that are special for~$i$, and set $R_i := S_i \setminus S^*_i$. We consider two cases, corresponding to conditions~$(b)$ and~$(c)$ of Definition~\ref{def:exploration:tree}, respectively.

\medskip
\noindent \textbf{Case 1:} There exists $i \in I$ with $|R_i| < \eps \big( |S_i| - 1 \big) + 1$.
\medskip

In this case we define 
$$\FF := \bigcup_{s \in S^*_i} \big\{ H \in \cA : H \text{ is a witness for } (s,i) \big\}.$$
Since a hyperplane $H$ cannot witness $(s,i)$ for more than one element $s \in S^*_i$, we have 
$$|\FF| \geq |S^*_i| \geq |S_i| - \eps \big( |S_i| - 1 \big) - 1 = (1-\eps)\big( |S_i| - 1 \big),$$ 
and by definition $i \in F(H)$ and $\mu_{I \setminus \{i\}}(H) > \delta$ for each $H \in \FF$.

\medskip
\noindent \textbf{Case 2:} $|R_i| \ge \eps \big( |S_i| - 1 \big) + 1$ for every $i \in I$. 
\medskip

In this case we shall apply Lemma~\ref{lem:localLem} to the set $R_I := \prod_{i \in I} R_i$ with $\eta = 1/6$. Define $\cA' \subseteq \cA$ by removing all hyperplanes that are witness for $(s,i)$ for some $i \in I$ and $s \in S_i$. Observe that none of the witness hyperplanes intersects $R_I$, so $\cA'':= \{ H \cap R_I : H \in \cA'\}$ is a cover of $R_I$. We claim that there exists a coordinate $i \in I$ such that 
\begin{equation}\label{equ:HeavyCoordinate} 
\sum_{H \in \cA' \,:\, i \in F(H)} e^{\eta |F(H)|} \tilde{\mu}\big( H \cap R_I \big) \geq \frac{\eta}{2},
\end{equation} 
where $\tilde{\mu}$ denotes the uniform measure on $R_I$. Since $\cA''$ is a cover of $R_I$, and noting that $|R_i| \ge 2$ for every $i \in I$ (by assumption, and since $|S_i| \ge 2$), this will follow from Lemma~\ref{lem:localLem} if we show that $\cA''$ is a collection of non-trivial hyperplanes in $R_I$. To do so, suppose for a contradiction that $R_I \subseteq H$ for some $H \in \cA'$, and observe that therefore $F(H) \cap I = \emptyset$, and hence also $S_I \subseteq H$. However, since $\cA$ is a minimal cover of $S_I$, this implies that $\cA = \{H\}$, and hence $F(\cA) = \emptyset$. This contradicts our assumption that $F(\cA) = I$, and thus each hyperplane in $\cA''$ is indeed non-trivial. 
As observed above, it follows by Lemma~\ref{lem:localLem} 
that~\eqref{equ:HeavyCoordinate} holds for some $i \in I$, as claimed.

Fix such an $i \in I$, and define $\cG := \cG^{(1)} \cup \cdots \cup \cG^{(k)}$, where
$$\cG^{(\ell)} := \big\{ H \in \cA' \,:\, i \in F(H) \textup{ and } |F(H)| = \ell \big\}$$
for each $\ell \in [k]$. Observe that $\cG^{(1)} = \emptyset$, since if $H \in \cA$ and $F(H) = \{i\}$, then $\mu_{I \setminus \{i\}}(H) = 1$, and so $H$ would have been removed when we formed $\cA'$. Similarly, for each $H \in \AA'$ with $i \in F(H)$ we have 
$$\mu_{I \setminus \{i\}}(H) \le \delta < 2^{-9} \lambda^2 \eps^{2\log_2 (1/\lambda \eps) + 11},$$ 
since otherwise $H$ would witness $(s,i)$ for some $s \in S_i$, and so would been removed when we formed $\cA'$. It follows, by Lemma~\ref{lem:MassOnRestrictedSpace}, that
$$\tilde{\mu}\big( H \cap R_I \big) \le \frac{\lambda}{2^{|F(H)|/2 + 4} |S_i|}$$
for every $H \in \cA'$ with $i \in F(H)$. Finally, combining this with~\eqref{equ:HeavyCoordinate} gives 
$$\sum_{\ell \geq 2} |\cG^{(\ell)}| \frac{e^{\eta \ell}}{2^{\ell/2}} = \sum_{H \in \cA' \,:\, i \in F(H)} \frac{e^{\eta |F(H)|}}{2^{|F(H)|/2} } \ge \sum_{H \in \cA' \,:\, i \in F(H)} e^{\eta|F(H)|} \tilde\mu(H\cap R_I)\frac{16|S_i|}{\lambda} \ge \frac{8\eta |S_i|}{\lambda},$$
and hence
$$\sum_{H \in \cG} 2^{- |F(H)|/4} = \sum_{\ell \geq 2} \frac{|\cG^{(\ell)}|}{2^{\ell/4}} = \sum_{\ell \geq 2} |\cG^{(\ell)}| \frac{e^{\eta \ell}}{2^{\ell/2}} \ge \frac{8\eta |S_i|}{\lambda} \ge \frac{|S_i|}{\lambda},$$
as required by~\eqref{eq:BigGarbage}.
\end{proof}

\enlargethispage*{\baselineskip}

The deduction of Lemma~\ref{lem:explorationTree} is now straightforward. Let $\p(T)$ denote the set of vertices of a rooted tree $T$ with no out-neighbours, and call $\p(T)$ the \emph{boundary} of $T$. 

\begin{proof}[Proof of Lemma~\ref{lem:explorationTree}]
We construct $\cT$, our exploration tree, inductively, with Lemma~\ref{one:step:exploration} providing the induction step. We begin our induction by defining $T_0$ to be a single vertex $v$, and setting $I_v := [k]$, and $\cA_v := \cA$. 
For the induction step, suppose that we have constructed a rooted tree $T_t$ (with root~$v$), a set $\emptyset \ne I_u \subseteq [k]$ and a collection $\cA_u \subseteq \cA$ for each vertex $u \in V(T_t)$, and an index $i_u \in I_u$ for each non-boundary vertex $u \in V(T_t) \setminus \p(T_t)$, such that condition~$(a)$ of Definition~\ref{def:exploration:tree} holds for every vertex $u \in V(T_t)$, and condition~$(ii)$ of Definition~\ref{def:index:tree}, and either condition~$(b)$ or~$(c)$ of Definition~\ref{def:exploration:tree}, hold for every non-boundary vertex $u \in V(T_t) \setminus \p(T_t)$. Observe that, since $\cA$ is a minimal cover of $S_{[k]}$ with hyperplanes such that $F(\cA) = [k]$, these conditions are satisfied in the base case $t = 0$.

To construct $T_{t+1}$, choose a vertex $u \in \p(T)$ such that $|I_u| \ge 2$, if one exists (we will deal with the other case below), and apply Lemma~\ref{one:step:exploration} to the set $I_u$ and minimal cover $\cA'_u = \{ H_{I_u} : H \in \cA_u \}$ of $S_{I_u}$ (noting that $F(\cA'_u) = I_u$, by the induction hypothesis). We obtain an index $i \in I_u$, a map $J \colon S_i \to \P(I \setminus \{i\})$ and a collection $\{ \cA_s : s \in S_i \}$ as in part~$(a)$ of the lemma, and either a collection $\FF \subseteq \cA'_u$ as in part~$(b)$, or a collection $\cG \subseteq \cA'_u$ as in part~$(c)$. In either case, we set $i_u := i$, add an out-neighbour of $u$ for each element $s \in S_i$ such that $J(s) \ne \emptyset$, and to the new vertex $w(s)$ corresponding to $s$ we assign the set $I_{w(s)} = J(s)$, the element $s_{w(s)} := s$, and the collection of hyperplanes 
$$\cA_{w(s)} := \big\{ H \in \cA_u : H_{I_u} \in \cA_s \big\}.$$ 

Observe first that condition~$(a)$ of Definition~\ref{def:exploration:tree} holds for each vertex $w(s)$. Indeed, we have $\cA_{w(s)} \subseteq \cA_u \subseteq \cA$, and it follows from Lemma~\ref{one:step:exploration}$(a)$ that the family 
$$\cA_{w(s)}' = \{ H_{J(s)} : H \in \cA_{w(s)} \} = \{ H_{J(s)} : H \in \cA_s \}$$ 
is a minimal cover of $S_{J(s)}$ with $F(\cA'_{w(s)}) = F(\cA_s) \setminus \{i\} = J(s)$, that $H_{i_u} \in \{s,S_{i_u}\}$ for each $H \in \cA_{w(s)}$, and that
$$F(\cA_{w(s)}) \subseteq \big( F(\cA_u) \setminus I_u \big) \cup F(\cA_s) \subseteq \big\{ i_w : w \prec_\cT u \big\} \cup J(s).$$
Note also that, by~\eqref{eq:J:union}, $u$ satisfies condition~$(ii)$ of Definition~\ref{def:index:tree}.

To complete the induction step, it remains to observe that $u$ is either good or bad, i.e., satisfies either condition~$(b)$ or~$(c)$ of Definition~\ref{def:exploration:tree}. Indeed, if Lemma~\ref{one:step:exploration}$(b)$ holds then  set 
$$\FF_u := \big\{ H \in \cA_u : H_{I_u} \in \FF \big\},$$ 
and if Lemma~\ref{one:step:exploration}$(c)$ holds then set 
$$\cG_u := \big\{ H \in \cA_u : H_{I_u} \in \cG \big\}.$$
In each case, the properties guaranteed by the lemma are exactly those that we require. Since the properties required of all vertices of $T_{t+1}$ other than $u$ and its out-neighbours continue to hold, it follows that $T_{t+1}$ satisfies the same properties that we assumed for $T_t$. 

Finally, observe that in passing from $T_t$ to $T_{t+1}$ we replace a vertex of the boundary by a finite number of boundary elements, each associated with with strictly smaller sets. This process must therefore eventually end, and when it does, it follows that $|I_u| = 1$ for every boundary vertex $u \in \p(T_t)$. When this happens, we simply set $i_u$ equal to the unique member of $I_u$ for each $u \in \p(T_t)$, and claim that $u$ is good. Indeed, by the induction hypothesis, the collection $\cA'_u$ forms a minimal cover of $S_{i_u}$ with $i_u \in F(\cA'_u)$. It follows that $\cA'_u$ consists of exactly $|S_{i_u}|$ singleton hyperplanes, and so~\eqref{cond:BigFrame} holds with $\FF_u := \cA_u$. Since condition~$(ii)$ of Definition~\ref{def:index:tree} holds automatically (with both sides equal to the empty set), it follows that the tree $\cT$ that we have constructed is an $(\lambda,\eps,\delta)$-exploration tree of $\cA$, as required.
\end{proof}

\section{Extraction of the frame, and the proof of Theorem~\ref{thm:mainStructureThm}}\label{frame:structure:sec}

\enlargethispage*{\baselineskip}

In order to prove Theorem~\ref{thm:mainStructureThm}, we will use the exploration tree $\cT$ constructed in the previous section, together with the bound~\eqref{eq:structure:assumption}, to find a $\delta$-generalized frame for $\cA$. Roughly speaking, we would like to do this by choosing $k$ vertices $\beta(1), \ldots, \beta(k)$ such that the label of $\beta(i)$ in $\cT$ is $i$, define a tree $T$ to be the union of the paths (in $\cT$) from $\beta(i)$ to the root, and define the elements $s_j(i)$ using the elements $s_u$. 
For each good vertex $\beta(i)$ we have a collection $\FF_i$ of hyperplanes more or less as required, and for each bad vertex we obtain a large collection of `garbage' hyperplanes. We might therefore hope to use~\eqref{eq:structure:assumption} to show that there are few bad vertices, and thus to deduce the bound~\eqref{eq:structure:conclusion}. 

There are two main problems with the strategy described above: the frame elements obtained for good vertices might not be disjoint, and each hyperplane might be included in the garbage set $\cG_u$ for a very large number of bad vertices. We overcome both obstacles in the same way: by choosing the vertices $\beta(i)$ via a depth-first search algorithm. We do not expect the reader to be able to immediately see why this choice should help in either case, but it turns out that proving that it does is (in both cases) surprisingly simple. 

In Section~\ref{delta:tree:def:sec} we will state precisely the object we will construct, and show that its existence implies the existence of a $\delta$-generalized frame. In Section~\ref{depth:first:sec} we will describe how we choose the sub-tree $T \subseteq \cT$, the frame elements $(\FF_1,\ldots,\FF_k)$, and the `garbage' sets $(\cG_1,\ldots,\cG_k)$; in Section~\ref{disjointness:sec}, we will prove two lemmas on the disjointness of the frame elements and garbage sets; and in Section~\ref{structure:proof:subsec} we will complete the proof of Theorem~\ref{thm:mainStructureThm}. 

\subsection{Tree-frames}\label{delta:tree:def:sec}

The purpose of this section is to introduce the following somewhat complicated objects, which also provide significantly more information (though we will not need this) about the covering system. We will use these objects to construct our frames.  

\begin{definition}\label{def:generalized:tree:frame}
Let $T$ be a rooted tree equipped with maps 
$$\alpha \colon V(T) \to [k], \qquad \beta \colon [k] \to V(T) \qquad \text{and} \qquad \gamma \colon E(T) \to S_1 \cup \cdots \cup S_k$$
such that
\begin{itemize}
\item[$(a)$] $\alpha(u) \ne \alpha(v)$ if $u \prec_T v$ and $u \ne v$;
\item[$(b)$] $\alpha\big( \beta( i ) \big) = i$ for each $i \in [k]$; 
\item[$(c)$] if $e \in E(T)$ and $v$ is the endpoint of $e$ that is closer to the root, then $\gamma(e) \in S_{\alpha(v)}$; 
\item[$(d)$] there exists a permutation $\pi$ of $[k]$ such that, if for each $i \in [k]$ we set 
$$J(i) := \big\{ \alpha(u) : u \prec_T \beta(i) \big\} \qquad \text{and} \qquad I(i) := [k] \setminus J(i),$$
then $J(\pi(i)) \subseteq \{\pi(1),\ldots,\pi(i)\}$ for each $i \in [k]$. 
\end{itemize}
 
Now, for each $\delta > 0$, a~\emph{$\delta$-generalized tree-frame} centred at $T$ is a sequence $(\FF_1,\ldots,\FF_k)$, where $\FF_i$ is a collection of at most $|S_i| - 1$ hyperplanes, satisfying
\begin{itemize} 
\item[$(i)$] $i \in F(H)$ for each $H \in \FF_i$.
\item[$(ii)$] $\mu_{I(i)}(H) > \delta$ for each $H \in \FF_i$. 
\item[$(iii)$] $H_j \in \{ \gamma(e), S_j \}$ for each $H \in \FF_i$ and each $j \in J(i) \setminus \{i\}$, where $e \in E(T)$ is the edge leaving the unique vertex $v \prec_T \beta(i)$ with $\alpha(v) = j$ in the direction of $\beta(i)$;   
\item[$(iv)$] If $\min\big\{ |S_i|, |S_j| \big\} \ge \delta^{-1}$ and $i \ne j$, then $\FF_i$ and $\FF_j$ are disjoint. 
\end{itemize}
\end{definition}

In Sections~\ref{depth:first:sec}--\ref{structure:proof:subsec} we will construct, for any  
$\cA$ satisfying~\eqref{eq:structure:assumption}, a $\delta$-generalized tree-frame satisfying~\eqref{eq:structure:conclusion}. The next lemma shows that this will be sufficient to prove Theorem~\ref{thm:mainStructureThm}.

\enlargethispage*{\baselineskip}

\begin{lemma}\label{lem:tree:frames:are:frames:too}
If $(\FF_1,\ldots,\FF_k)$ is a $\delta$-generalized tree-frame centred at a rooted tree $T$, then $(\FF_1,\ldots,\FF_k)$ is a $\delta$-generalized frame.
\end{lemma}

\begin{proof}
Let $T$ be a rooted tree equipped with maps $\alpha$, $\beta$ and $\gamma$ satisfying conditions~$(a)$--$(d)$ of Definition~\ref{def:generalized:tree:frame}. In particular, let $\pi$ be the permutation given by condition~$(d)$, and (to simplify the notation) let us permute the sets $S_1,\ldots,S_k$ so that $\pi$ is the identity, and therefore $J(i) \subseteq \{1,\ldots,i\}$ (and hence $I(i) \supseteq \{ i+1,\ldots,k \}$) for each $i \in [k]$. Now, for each $i \in [k]$ and $j \in J(i) \setminus \{i\}$, set $s_j(i) := \gamma(e) \in S_j$, where $e \in E(T)$ is the edge leaving the unique vertex $v \prec_T \beta(i)$ with $\alpha(v) = j$ in the direction of $\beta(i)$.

We claim that, for each $H \in \FF_i$,
$$i \in F(H), \qquad\quad \mu_{I(i)}(H) > \delta \qquad\quad \text{and} \qquad\quad H_j \in \big\{ s_j(i), S_j \big\}.$$
Indeed, these follow directly from properties~$(i)$,~$(ii)$ and~$(iii)$ of Definition~\ref{def:generalized:tree:frame}. Finally, observe that if $\min\big\{ |S_i|, |S_j| \big\} \ge \delta^{-1}$ then $\FF_i$ and $\FF_j$ are disjoint, by property~$(iv)$.
\end{proof}

\subsection{Constructing the frame}\label{depth:first:sec}

{\setstretch{1.07}

In this section we will construct the $\delta$-generalized tree-frame $(\FF_1,\ldots,\FF_k)$, along with the rooted tree $T$, and a collection $(\cG_1,\ldots,\cG_k)$ of `garbage' sets. Let $C > 0$ and $\eps > 0$ be arbitrary, as in the statement of Theorem~\ref{thm:mainStructureThm}, and set
\begin{equation}\label{def:alpha:delta}
\lambda := \frac{\eps}{2^4 C} \qquad \text{and} \qquad \delta := 2^{-23} \lambda^4 \eps^{2\log_2 (1/\lambda \eps) + 15}.
\end{equation}
Recall that the sets $S_1,\ldots,S_k$ were fixed earlier, and let us fix, for the rest of this section, a minimal cover $\cA$ of $S_{[k]}$ with hyperplanes such that $F(\cA) = [k]$. By Lemma~\ref{lem:explorationTree}, there exists an $(\lambda,\eps/2,\delta)$-exploration tree of $\cA$; let us also fix such a tree $\cT$. 

The first step is to observe that every $i \in [k]$ occurs as the label $i_u$ of some vertex $u \in V(\cT)$. This is an immediate consequence of the following simple observation about index trees.

\begin{observation}\label{obs:index:voracious}
Let $\cT$ be an index tree, let $u \in V(\cT)$, and let $j \in I_u$. Then there exists $v \in V(\cT)$, with $u \prec_\cT v$, such that $i_v = j$. 
\end{observation}

\begin{proof}
This follows easily from~\eqref{eq:index:tree}: if $j \in I_v$ and $i_v \ne j$, then $j \in I_w$ for some $w \in N(u)$, and if $j \in I_v$ and $v \in \p(\cT)$ then $I_v = \{j\}$, so $i_v = j$. 
\end{proof}

\enlargethispage*{\baselineskip}

To extract our $\delta$-generalized tree-frame from the exploration tree $\cT$, we will also need the notion of a \emph{depth-first search ordering} $\prec$ on the vertices of a rooted tree $T$. This is defined by placing an arbitrary linear order on the out-neighbours of each vertex of $T$, and then setting $u \prec v$ if either $u \prec_T v$, or if the branch leading to $u$ precedes the branch leading to $v$ in the ordering of the neighbours of the last common ancestor of $u$ and $v$. 

\begin{definition}\label{def:final:frame}
Let $\prec$ be a depth-first search ordering on the vertices of $\cT$. We define a rooted tree~$T$ and a $\delta$-generalized tree-frame centred at $T$ as follows: 
\begin{itemize}
\item[$1.$] For each $i \in [k]$, define $\beta(i)$ to be the $\prec$-minimal vertex $u$ of $\cT$ such that $i_u = i$. \smallskip
\item[$2.$] Define $T$ to be the union of the paths in $\cT$ from $\beta(1), \ldots, \beta(k)$ to the root.\smallskip 
\item[$3.$] For each $u \in V(T)$, define $\alpha(u) := i_u$.\smallskip
\item[$4.$] For each edge $uv \in E(T)$, where $u \in N(v)$, define $\gamma(uv) := s_u \in S_{i_v}$.\smallskip 
\item[$5.$] 
\begin{itemize}
\item[$(a)$] Set $\FF_i := \FF_{\beta(i)}$ and $\cG_i := \emptyset$ for each $i \in [k]$ such that $\beta(i)$ is a good vertex of $\cT$. \vskip0.05cm
\item[$(b)$] Set $\cG_i := \cG_{\beta(i)}$ and $\FF_i := \emptyset$ for each $i \in [k]$ such that $\beta(i)$ is a bad vertex of $\cT$.
\end{itemize}
\end{itemize}
\end{definition}

Properties $(a)$--$(d)$ of Definition~\ref{def:generalized:tree:frame} follow easily from this construction. Indeed, we have $\alpha(\beta(i)) = i$ for each $i \in [k]$ by our choice of $\alpha$ and $\beta$, and $\gamma(uv) \in S_{i_v} = S_{\alpha(v)}$ for all $uv \in E(T)$ with $u \in N(v)$. To see that $\alpha(u) \ne \alpha(v)$ if $u \prec_T v$ and $u \ne v$, recall $\cT$ is an index tree, and therefore satisfies~\eqref{eq:index:tree}, so $\alpha(u) = i_u$ is not included in any of the sets associated with the descendants of $u$. For $(d)$, let $\pi$ be the permutation of $[k]$ given by the ordering $\prec$ restricted to $\{\beta(1),\ldots,\beta(k)\}$, and observe that if $u \prec_T \beta(\pi(i))$ then $u \prec \beta(\pi(i))$, and hence $\beta(\alpha(u)) \prec \beta(\pi(i))$, by our choice of $\beta$. It follows that $\alpha(u) \in \{\pi(1),\ldots,\pi(i)\}$, and therefore that $J(\pi(i)) = \big\{ \alpha(u) : u \prec_T \beta(\pi(i)) \big\} \subseteq \{\pi(1),\ldots,\pi(i)\}$
for each $i \in [k]$, as required.
}
 
{\setstretch{1.09}

The following lemma shows that properties $(i)$, $(ii)$ and $(iii)$ of Definition~\ref{def:generalized:tree:frame} also hold.
 
\begin{lemma}\label{lem:basic:frame:properties}
Let $i \in [k]$ and $H \in \FF_i$, let $j \in J(i) \setminus \{i\}$, and let $e_j \in E(T)$ be the edge leaving the unique vertex $v \prec_T \beta(i)$ with $\alpha(v) = j$ in the direction of the vertex $\beta(i)$. Then 
$$i \in F(H), \qquad \mu_{I(i)}(H) > \delta \qquad \text{and} \qquad H_j \in \{ \gamma(e_j), S_j \}.$$
\end{lemma}

\begin{proof}
Note that $\FF_i \ne \emptyset$ implies that the vertex $\beta(i)$ is good. By Definition~\ref{def:exploration:tree}$(b)$, it follows that $i \in F(H)$, and also that $\mu_{I_{\beta(i)} \setminus \{i\}}(H) > \delta$. Moreover $\FF_i =  \FF_{\beta(i)} \subseteq \cA_{\beta(i)}$ and $F(\cA_{\beta(i)}) \subseteq J(i) \cup I_{\beta(i)}$, by Definition~\ref{def:exploration:tree}, and therefore $H_j = S_j$ for every $j \in I(i) \setminus I_{\beta(i)}$. Since $i \in J(i)$, and therefore $i \not\in I(i)$, it follows that $\mu_{I(i)}(H) > \delta$.

Now suppose that $e_j = uv$, with $u \in N(v)$, and observe that, by Definition~\ref{def:exploration:tree}$(a)$, we have $H'_j \in \{s_u,S_j\} = \{ \gamma(e_j), S_j\}$ for every hyperplane $H' \in \cA_u$. Noting that $\FF_i \subseteq \cA_{\beta(i)} \subseteq \cA_u$ (by Definition~\ref{def:exploration:tree}, and since $u \prec_T \beta(i)$), it follows that $H_j \in \{\gamma(e_j),S_j\}$, as claimed.
\end{proof}

 It therefore only remains to show that property~$(iv)$ of Definition~\ref{def:generalized:tree:frame} and the inequality~\eqref{eq:structure:conclusion} hold. Both of these properties will follow from our choice of $\beta(1), \ldots, \beta(k)$.

\subsection{Lemmas on disjointness}\label{disjointness:sec}

In this section we will prove two straightforward but crucial lemmas; the first verifies condition~$(iv)$ of the definition of a $\delta$-generalized tree-frame. 

\begin{lemma}\label{lem:big:frames:disjoint}
If\/ $\min \{ |S_i|, |S_j| \} \ge \delta^{-1}$ and $i \ne j$, then $\FF_{i}$ and $\FF_{j}$ are disjoint. 
\end{lemma}

\begin{proof}
Suppose that $H \in \FF_{\beta(i)} \cap \FF_{\beta(j)}$, and suppose that $\beta(i) \prec \beta(j)$ in the depth-first search ordering. Recall that $i,j \in F(H)$, by Definition~\ref{def:exploration:tree}$(b)$, and $\mu_{I(i)}(H) > \delta$, by Lemma~\ref{lem:basic:frame:properties}. Since $|S_j| \ge \delta^{-1}$, it follows that $j \not\in I(i)$, and hence $j \in J(i)$, i.e., there exists $u \prec_T \beta(i)$ with $\alpha(u) = j$. However, this is a contradiction, since $u \prec \beta(i) \prec \beta(j)$ in the depth-first search ordering, and $\beta(j)$ was chosen to be the $\prec$-minimal vertex $u$ of $\cT$ such that $\alpha(u) = j$.
\end{proof}

The final lemma we need shows that each garbage set only appears on a single path through~$T$. Since the number of fixed coordinates of $H_{I_u}$ decreases along the path (and decreases strictly whenever $\alpha(u) \in F(H)$), this will imply that each hyperplane contributes only $O(1)$ to the sum of the left-hand side of~\eqref{cond:BigGarbage} over vertices $u \in \{ \beta(1), \ldots, \beta(k) \}$. 

\begin{lemma}\label{lem:garbage:on:a:path}
Let $H \in \cG_i \cap \cG_j$. Then either $\beta(i) \prec_{\cT} \beta(j)$, or $\beta(j) \prec_{\cT} \beta(i)$.
\end{lemma}

\begin{proof}
Suppose (without loss of generality) that $\beta(i) \prec \beta(j)$ in the depth-first search ordering, and suppose that $\beta(i) \not\prec_{\cT} \beta(j)$, which implies that $u \prec \beta(j)$ for every $u \in V(\cT)$ with $\beta(i) \prec_{\cT} u$. Note that $j \in F(H) \subseteq I_{\beta(i)} \cup J(i)$, since $H \in \cG_i \cap \cG_j$ and by Definition~\ref{def:exploration:tree}. 

If $j \in J(i)$, then $i_v = j$ for some $v \in V(\cT)$ with $v \prec_{\cT} \beta(i)$, and hence $v \prec \beta(i) \prec \beta(j)$. On the other hand, if $j \in I_{\beta(i)}$, then by Observation~\ref{obs:index:voracious} we have $i_v = j$ for some $v \in V(\cT)$ with $\beta(i) \prec_{\cT} v$, and hence (by the observation above) $v \prec \beta(j)$. In either case, this contradicts our choice of $\beta(j)$ as the $\prec$-minimal vertex $v$ of $\cT$ such that $i_v = j$.
\end{proof}

\enlargethispage*{\baselineskip}

Let us record here the following simple consequence of Lemma~\ref{lem:garbage:on:a:path}.

\begin{lemma}\label{lemma:different:fixed:sets:along:path}
If $H \in \cG_i \cap \cG_j$ and $i \ne j$, then $|F(H) \cap I_{\beta(i)}| \ne |F(H) \cap I_{\beta(j)}|$.
\end{lemma}

\begin{proof}
By Lemma~\ref{lem:garbage:on:a:path}, we have (without loss of generality) $\beta(i) \prec_{\cT} \beta(j)$, which implies, by~\eqref{eq:index:tree} and since $\beta(i) \ne \beta(j)$, that $I_{\beta(j)} \subseteq I_{\beta(i)} \setminus \{i\}$. Since $i \in F(H) \cap I_{\beta(i)}$, it follows that $F(H) \cap I_{\beta(j)} \subsetneq F(H) \cap I_{\beta(i)}$, as required.  
\end{proof}

}

\subsection{The proof of Theorem~\ref{thm:mainStructureThm}}\label{structure:proof:subsec}

We are now ready to prove our main structural result, Theorem~\ref{thm:mainStructureThm}. It only remains to show that the inequality~\eqref{eq:structure:conclusion} follows from~\eqref{eq:structure:assumption}. We will use the following easy consequence of Lemma~\ref{lemma:different:fixed:sets:along:path}. Set $B := \{ i \in [k] : \beta(i) \text{ is bad}\}$. 

\begin{lemma}\label{lem:total:garbage:bound}
$$\bigg| \bigcup_{i \in B} \cG_i \bigg| \, \ge \, \frac{1}{5\lambda} \sum_{i \in B} |S_i|.$$ 
\end{lemma}

\begin{proof}
Summing~\eqref{cond:BigGarbage} over $i \in B$, we obtain 
$$\sum_{i \in B} \sum_{H \in \cG_i} 2^{-|F(H) \cap I_{\beta(i)}|/4} \ge \, \frac{1}{\lambda} \sum_{i \in B} |S_i|.$$
Now, by Lemma~\ref{lemma:different:fixed:sets:along:path}, for each $H$ and $\ell \ge 2$ there is at most one value of $i \in B$ such that $H \in \cG_i$ and $|F(H) \cap I_{\beta(i)}| = \ell$, so for each $H \in \bigcup_{i \in B} \cG_i$ we have 
$$\sum_{i \in B \,:\, H \in \cG_i} 2^{-|F(H) \cap I_{\beta(i)}|/4} \le \sum_{\ell = 2}^\infty 2^{-\ell/4} = \big( 2^{1/2} - 2^{1/4} \big)^{-1} < 5,$$
as required.
\end{proof}

Theorem~\ref{thm:mainStructureThm} now follows easily from the lemmas above.

\begin{proof}[Proof of Theorem~\ref{thm:mainStructureThm}]
We claim that the sequence $(\FF_1,\ldots,\FF_k)$ constructed in Definition~\ref{def:final:frame} is a $\delta$-generalized tree-frame centred at $T$, and satisfies~\eqref{eq:structure:conclusion}. By Lemma~\ref{lem:tree:frames:are:frames:too}, it will follow that $(\FF_1,\ldots,\FF_k)$ is also a $\delta$-generalized frame, so this will be sufficient to prove the theorem. Note that properties $(a)$--$(d)$ and $(i)$--$(iv)$ of Definition~\ref{def:generalized:tree:frame} follow from the comments after Definition~\ref{def:final:frame}, and by Lemmas~\ref{lem:basic:frame:properties} and~\ref{lem:big:frames:disjoint}. Moreover, by discarding excess hyperplanes if necessary, we may assume that $|\FF_i| \le |S_i| - 1$ for each $i \in [k]$. It therefore only remains to show that~\eqref{eq:structure:conclusion} holds.  

To do so, recall that $\cT$ is an $(\lambda,\eps/2,\delta)$-exploration tree of $\cA$, and hence 
$$|\FF_i| > \big( 1 - \eps/2 \big) \big( |S_i| - 1 \big)$$
for each $i$ such that $\beta(i)$ is a good vertex, i.e., for each $i \in [k] \setminus B$. Now, by Lemma~\ref{lem:total:garbage:bound} and the condition~\eqref{eq:structure:assumption}, we have 
$$\frac{1}{5\lambda} \sum_{i \in B} |S_i| \le |\cA | \leq C \sum_{i=1}^k \big( |S_i| - 1 \big),$$
and hence, recalling from~\eqref{def:alpha:delta} that $\lambda = \eps / (2^4 C)$, we obtain 
\begin{align*}
\sum_{i=1}^k |\FF_i| & \, \ge \sum_{i \in [k] \setminus B} \big( 1 - \eps/2 \big) \big( |S_i| - 1 \big) \, \ge \big( 1 - \eps/2 \big) \sum_{i=1}^k \big( |S_i| - 1 \big) - \sum_{i \in B} |S_i| \\
& \, \ge \, \big( 1 - \eps/2 - 5C\lambda \big) \sum_{i=1}^k \big( |S_i| - 1 \big) \, \ge \, (1 - \eps) \sum_{i=1}^k \big( |S_i| - 1 \big),
\end{align*}
as required. 
\end{proof}

\section{Arithmetic frames and the proof of the lower bound}\label{arithmetic:sec}

In order to deduce Theorem~\ref{thm:MainCountingThm} from Theorem~\ref{thm:mainStructureThm}, we will need to bound the number of $\delta$-generalized frames in the integers. In this section we will warm up for the calculation ahead by counting a simpler set of objects, which we call `arithmetic frames', and thereby deducing a lower bound on the number of minimal covering systems of $\Z$. Recall that\footnote{We remark that the constant $\tau$ also appears in the study of the iterated divisor function, see~\cite{BEFS,ETT}.}
\[ 
\tau = \sum_{t = 1}^{\infty} \left(\log \frac{t+1}{t} \right)^2 \approx 0.977. 
\]
The following proposition provides the lower bound in Theorem~\ref{thm:MainCountingThm}. 

\begin{prop}\label{lem:NumbOfFrames}
The number of minimal covering systems of $\Z$ of size $n$ is at least 
$$\exp\left( \left( \frac{4\sqrt{\tau}}{3} + o(1)\right)\frac{n^{3/2}}{(\log n)^{1/2}} \right)$$
as $n \to \infty$.
\end{prop}

We shall first prove Proposition~\ref{lem:NumbOfFrames} for an infinite sequence of values of $n$ (see~\eqref{eq:n:and:N}); since this sequence will be sufficiently dense, it will then be easy to deduce the bound for the remaining values of $n$. For each $n$ in our sequence, we will choose a single value of $N$, and count only covering systems $\cA$ of size $n$ with $\lcm(\cA) = N$. We will moreover count only covering systems that correspond to simple frames of a certain family of sets (see below), with a specific (carefully chosen) order, see Definition~\ref{def:the:lower:bound:order}. We remark that when $N$ is not square-free, this is not quite as straightforward as counting the simple frames, since there will exist hyperplanes that do not correspond to arithmetic progressions in $\Z_N$. In order to characterise the hyperplanes that do, we need to introduce a little notation. 

Given $N = p_1^{\g_1} \cdots p_m^{\g_m} > 1$, we define 
\begin{equation}\label{def:bracketsN}
\langle N \rangle := \bigcup_{i=1}^m \big\{ (p_i,j) : j \in [\g_i] \big\},
\end{equation}
and set $S_{(p,e)} = \{0,\ldots,p-1\}$ for each $(p,e) \in \< N \>$. Now define a map
$$\varphi_N \colon \Z_N \to S_{\< N \>} = \prod_{(p,e) \in \< N \>} S_{(p,e)}$$
as follows: if $x \in \Z_N$, then $y = \varphi_N(x) \in S_{\< N \>}$ is the vector such that, for each $(p,e) \in \< N \>$, $y_{(p,e)} \in S_{(p,e)}$ is the coefficient of $p^{e-1}$ in the $p$-ary expansion of $x$ modulo $p^e$. Observe that $\varphi_N$ is a bijection, by the Chinese Remainder Theorem. 

We say that a hyperplane $H$ in $S_{\< N \>}$ is \emph{arithmetic} if $\varphi_N^{-1}(H)$ is an arithmetic progression in $\Z_N$. The following observation provides a simple characterization of arithmetic hyperplanes. 



\begin{observation}\label{obs:arithmetic:hyperplanes}
A hyperplane $H$ in $S_{\< N \>}$ is arithmetic if and only if, for each prime $p$, the set 
$$\big\{ (p,e) \in \< N \> : (p,e) \in F(H) \big\}$$ 
forms a (possibly empty) initial segment of the sequence $(p,1), (p,2), (p,3), \ldots$
\end{observation}

{\setstretch{1.12}

\begin{proof}
Suppose first that $H$ is arithmetic, so $A := \varphi_N^{-1}(H)$ is an arithmetic progression in $\Z_N$. Let $d$ be the modulus of $A$, and observe that $(p,e) \in F(H)$ if and only if $p^e$ divides $d$, by the definition of $\varphi_N$. On the other hand, if $\{(p,1),\ldots,(p,e(p))\} \subset F(H)$ and $(p,e') \not\in F(H)$ for all $e' > e(p)$, then every pair of points of $\varphi_N^{-1}(H)$ differs by a multiple of $p^{e(p)}$, and therefore $\varphi_N^{-1}(H)$ is contained in an arithmetic progression with modulus $d = \prod_p p^{e(p)}$. Since $|H| = N/d$, it follows that $\varphi_N^{-1}(H)$ is in fact the entire arithmetic progression, as claimed.
\end{proof}

Let us now say that a total ordering $\prec$ on the elements of $\langle N \rangle$ is 
\emph{arithmetic} if 
\begin{equation}\label{def:arithmetic:ordering}
(p_i,1) \prec (p_i,2) \prec \cdots \prec (p_i,\gamma_i)
\end{equation}
for all $i \in [m]$. Note that~\eqref{def:arithmetic:ordering} does not impose any constraint on $\prec$ for different primes, and in particular we may have $(p,i) \prec (q,j) \prec (p,i+1)$. We say that a simple frame of $S_{\< N \>}$ is `arithmetic' if the order of the sets is arithmetic, and if moreover each of the hyperplanes of the frame is arithmetic. We can now prove the following lower bound on the number of minimal covering systems of $\Z$ of size $n$.

\begin{lemma}\label{lem:general:lower:bound}
Let $N = p_1^{\g_1} \cdots p_m^{\g_m} > 1$, and let $\prec$ be an arithmetic ordering of $\< N \>$. There are at least 
$$\exp\bigg( \sum_{(p,e) \in \langle N \rangle } (p-1) \sum_{\substack{(q,f) \prec (p,e) \\ q \ne p}} \log\left( \frac{f+1}{f} \right) \bigg)$$ 
minimal covering systems of $\Z$ of size $n := \sum_{i = 1}^m \gamma_i (p_i - 1) + 1$. 
\end{lemma}

\begin{proof}
To prove the lemma we count arithmetic frames of $S_{\< N \>}$ centred at $(0,\ldots,0)$, where the sets $S_{(p,e)}$ are listed in the order $\prec$. Recall from Definition~\ref{def:frame} that, for each $(p,e) \in \< N \>$ and each $a \in \{1,\ldots,p-1\}$, we need to choose an arithmetic hyperplane of the form
$$\big[ x_1,\ldots,x_{i-1},a,*,\cdots,*\big],$$
with $x_j \in \{ 0, * \}$ for each $j \in [i-1]$, where $(p,e)$ is the $i$th element in the ordering $\prec$. To do so, we will choose, for each prime $q \ne p$, an initial segment (in the order $\prec$) of the set 
$$J_q(p,e) := \big\{ (q,f) \in \< N \> : (q,f) \prec (p,e) \big\}$$ 
set $x_j = 0$ for the corresponding coordinates, and set $x_j = *$ for all other elements of $J_q(p,e)$. If we also set $x_j = 0$ for every $j \in J_p(p,e)$ then, by Observation~\ref{obs:arithmetic:hyperplanes}, every hyperplane obtained in this way will be arithmetic, and therefore the frame we construct will be arithmetic. We claim that each such choice gives a different minimal covering system of $\Z$ of size $n$. 

To see this, note that the frame consists of $n-1$ arithmetic hyperplanes, each of which corresponds (via the bijection $\varphi_N^{-1}$) to an arithmetic progression in $\Z_N$. Moreover, the only element of $\Z_N$ not covered by these arithmetic progressions is $0$, so adding this progression gives a covering system of $\Z$ of size $n$, and (as observed after Definition~\ref{def:frame}) if we remove the hyperplane $H = \big[ x_1,\ldots,x_{i-1},a,*,\cdots,*\big]$, then the element $(0,\ldots,0,a,0,\ldots,0)$ will be uncovered by the remaining hyperplanes, so the covering system we have constructed is minimal. Finally, each hyperplane in the frame has a unique entry $a \not\in \{ 0, * \}$, and therefore each choice leads to a distinct covering system, as claimed.

Finally, since we have exactly
$$\prod_{q \ne p} \big( |J_q(p,e)| + 1 \big) = \exp\bigg( \sum_{\substack{(q,f) \prec (p,e) \\ q \ne p}} \log\left( \frac{f+1}{f} \right) \bigg)$$ 
choices for each hyperplane corresponding to $(p,e)$, the lemma follows.
\end{proof}

}
{\setstretch{1.25}

Now, for each arithmetic ordering $\prec$ of $\< N \>$, let us define 
\begin{equation} \label{equ:defofQord}
Q(N, \prec) := \sum_{(p,e) \in \langle N \rangle } (p-1) \sum_{\substack{(q,f) \prec (p,e) \\ q \ne p}}  \log\left( \frac{f+1}{f} \right). 
\end{equation}
We will use the following particular arithmetic ordering $<$ to prove Proposition~\ref{lem:NumbOfFrames}.

\begin{definition}\label{def:the:lower:bound:order}
For each prime $p$ and integer $e \in \N$, set $y_{p,e} := (p-1)\big( \log\frac{e+1}{e} \big)^{-1}$. Now, given primes $p$ and $q$, and integers $e,f \in \N$, define 
$$(q,f) < (p,e) \qquad \Leftrightarrow \qquad y_{q,f} < y_{p,e}.$$
Moreover, if $x \in \R$ then we write $(p,e) < x$ if and only if $y_{p,e} < x$, and define
\begin{equation}\label{eq:n:and:N}
n(x) := 1 + \sum_{(p,e) < x} (p-1) \qquad \text{and} \qquad N(x) := \prod_{(p,e) < x} p.
\end{equation}
\end{definition}

Note that $n(x),N(x) < \infty$ for every $x \in \R$, and that for any $N \in \N$, the ordering $<$ on $\< N \>$ is arithmetic. Our next lemma, combined with Lemma~\ref{lem:general:lower:bound}, implies Proposition~\ref{lem:NumbOfFrames}. 

\begin{lemma}\label{lem:calculating:Q}
Let $x > 0$, and set $N = N(x)$ and $n = n(x)$. Then  
$$Q(N,<) = \left( \frac{4\sqrt{\tau}}{3} + o(1)\right) \frac{n^{3/2}}{(\log n)^{1/2}}$$ 
as $x \to \infty$.
\end{lemma}
 
\begin{proof}
Recalling the definition of $Q(N,<)$, observe first that, for each $(p,e) \in \< N \>$,
\begin{equation}\label{eq:fix:p:sum:over:f}
\sum_{\substack{(q,f) < (p,e) \\ q \ne p}} \log\left( \frac{f+1}{f} \right) = \sum_{f\geq 1} \left|\left\lbrace q \ne p : q - 1 < y_{p,e} \log\left( \frac{f+1}{f} \right) \right\rbrace\right| \cdot \log\left( \frac{f+1}{f} \right).
\end{equation}
Now, by the prime number theorem, for each fixed $f \in \N$ and as $y_{p,e} \to \infty$,
$$\left|\left\lbrace q \ne p : q - 1 < y_{p,e} \log\left( \frac{f+1}{f} \right) \right\rbrace\right| = \big( 1 + o(1) \big) \frac{y_{p,e} \log\big( \frac{f+1}{f} \big)}{\log\big( y_{p,e} \log\big( \frac{f+1}{f} \big) \big)}.$$
Moreover, the sum in~\eqref{eq:fix:p:sum:over:f} of the terms with $f \ge f_0$ is $o\big( y_{p,e} / \log y_{p,e} \big)$ as $f_0 \to \infty$, so
$$\sum_{\substack{(q,f) < (p,e) \\ q \ne p}} \log\left( \frac{f+1}{f} \right) = \frac{\big( 1 + o(1) \big) y_{p,e}}{\log y_{p,e}} \sum_{f \ge 1} \left(\log \frac{f+1}{f}\right)^2 = \big( \tau + o(1) \big) \frac{y_{p,e}}{\log y_{p,e}}$$
as $y_{p,e} \to \infty$. We next fix $e \in \N$, and sum over primes $p$. We obtain
\begin{align}
\sum_{p \,:\, y_{p,e} < x} (p-1) \sum_{\substack{(q,f) < (p,e) \\ q \ne p}} \log\left( \frac{f+1}{f} \right) \nonumber
& \, = \sum_{p - 1 \,<\, x\log\frac{e+1}{e}} \big( \tau + o(1) \big)  \frac{(p-1)^2}{\log \big( \frac{e+1}{e} \big) \cdot \log y_{p,e}}\\
& \, = \big( \tau + o(1) \big)\frac{x^3 \left(\log\frac{e+1}{e}\right)^2}{3(\log x)^2} \label{eq:fix:e:sum:over:p}
\end{align}
as $x \to \infty$, again using the prime number theorem.\footnote{Indeed, the prime number theorem implies that $\sum_{p < z} p^2 = \big(1/3 + o(1) \big) z^3 / \log z$ as $z \to \infty$.} Thus, summing over $e$, and noting that the left-hand side of~\eqref{eq:fix:e:sum:over:p} is uniformly bounded from above by an absolute constant times the right-hand side (without the $o(1)$ term), we obtain
$$Q(N,<) = \big( \tau + o(1) \big) \frac{x^3}{3(\log x)^2} \sum_{e \ge 1}\left(\log\frac{e+1}{e}\right)^2 = \big( \tau^2 + o(1) \big) \frac{x^3}{3(\log x)^2}$$
as $x \to \infty$. Finally, using the prime number theorem a third time, we obtain 
\begin{align}
n(x) & \, = 1 + \sum_{e \ge 1} \sum_{p - 1 \,<\, x\log\frac{e+1}{e}} (p - 1) = \big(1+o(1) \big)\sum_{e \ge 1} \frac{\left(x\log\frac{e+1}{e} \right)^2 }{2\log\left(x\log\frac{e+1}{e}\right)} \nonumber\\
& \, = \big(1+o(1) \big) \frac{x^2}{2\log x}\sum_{e \ge 1} \left(\log \frac{e+1}{e} \right)^2 = \big( \tau + o(1) \big)\frac{x^2}{2\log x},\label{eq:n:of:x}
\end{align}
and hence $Q(N,<) \cdot n^{-3/2} \sqrt{\log n} \to \sqrt{2} (\tau^2 / 3) (\tau/2)^{-3/2} = 4\sqrt{\tau} / 3$ as $x \to \infty$, as claimed.
\end{proof}

We can now easily deduce the lower bound in Theorem~\ref{thm:MainCountingThm}, the only remaining difficulty being to deal with those $n \in \N$ that are not of the form $n = n(x)$ for some $x \in \R$.

\begin{proof}[Proof of Proposition~\ref{lem:NumbOfFrames}]
It follows immediately from Lemmas~\ref{lem:general:lower:bound} and~\ref{lem:calculating:Q} that the number of minimal covering systems of $\Z$ of size $n(x)$ is at least 
$$e^{Q(N(x),<)} = \exp\left( \left( \frac{4\sqrt{\tau}}{3} + o(1)\right)\frac{n(x)^{3/2}}{(\log n(x))^{1/2}} \right)$$ 
as $x \to \infty$. Let $x > 0$ be maximal such that $n(x) \le n$, and set $t := n - n(x)$. Observe that $t < x = o(n)$, by~\eqref{eq:n:of:x}, and that, by removing the hyperplane $[0,\ldots,0]$ (i.e., the progression $\{ 0 \pmod N \}$) from the construction given in the proof of Lemma~\ref{lem:general:lower:bound}, we obtain a family of minimal covers of $\Z_N \setminus \{0\}$ of size $n(x) - 1$. We complete each to a minimal cover of $\Z$ of size  $n$ by adding the progressions $\big\{ 2^{\ell-1} N \pmod {2^\ell N} \big\}$, for each $\ell \in [t]$, and $\big\{ 0 \pmod {2^t N} \big\}$. We obtain a family of 
$$\exp\left( \left( \frac{4\sqrt{\tau}}{3} + o(1)\right)\frac{n^{3/2}}{(\log n)^{1/2}} \right)$$
minimal covering systems of $\Z$ of size $n(x) + t = n$, as required.
\end{proof}

}

\section{Counting coverings that are far from frames}\label{counting:weird:frames:sec}

{\setstretch{1.12}
In this section we will begin the deduction of Theorem~\ref{thm:MainCountingThm} from Theorem~\ref{thm:mainStructureThm} by bounding the number of minimal covers that fail to satisfy~\eqref{eq:structure:assumption}. In the process, we will obtain a short proof of weaker version of Theorem~\ref{thm:MainCountingThm}, bounding the number of minimal covering systems of $\Z$ of size $n$ up to a constant factor in the exponent.

\begin{prop}\label{lem:LargemCrudeBound} 
Let $C > 0$ be a constant, and let $n \in \N$ and $N = p_1^{\g_1} \cdots p_m^{\g_m}$ satisfy 
$$n > C \sum_{i=1}^m \gamma_i \big( p_i - 1 \big).$$
Then the number of minimal covering systems $\cA$ of $\Z$ of size $n$ with $\lcm(\cA) = N$ is at most
$$\exp\left( \left( \frac{2\sqrt{\tau}}{\sqrt{C}} + o(1)\right)\frac{n^{3/2}}{(\log n)^{1/2}} \right)$$
as $n \to \infty$.
\end{prop}

In order to bound the number of covering systems, we will need to bound the number of choices for the modulus $d$ and shift $a$ of each arithmetic progression in $\cA$. The following simple but important lemma, which we will use again later, shows that, given the moduli, we have relatively few choices for the shifts. 

\begin{lemma}\label{lem:counting:shifts} 
Let $d_1,\ldots,d_n \in \N$. There are at most $(n!)^2$ minimal covering systems $\cA = \{A_1,\ldots,A_n\}$ of\/ $\Z$ of size $n$ such that, for each $i \in [n]$, the modulus of $A_i$ is $d_i$.
\end{lemma}

\begin{proof}
Let $\cA = \{ A_1,\ldots,A_n \}$ be a minimal covering system of $\Z$, and observe that we may reorder the elements of $\cA$ so that, for each $i \in [n]$, the arithmetic progression $A_i$ covers at least a $1/i$ proportion of the set
$$R_i := \Z \setminus \bigcup_{j > i} A_j.$$
Indeed, to see that this is possible we simply choose the sets one by one (in reverse order), letting $A_i$ be the (remaining) progression in $\cA$ whose intersection with $R_i$ has largest density, observing that $R_i$ is non-empty (since $\cA$ is minimal) and recalling that $\cA$ covers $\Z$. The total number of choices for $\cA$ is therefore at most the sum over permutations of $(d_1,\ldots,d_n)$ of the number of sequences $(A_1,\ldots,A_n)$ with this additional property.  

Now let $i \in [n]$, and suppose that we have already chosen progressions $(A_{i+1},\ldots,A_n)$. We claim that we have at most $i$ choices for the arithmetic progression~$A_i$. Indeed, since the progressions $\{ a \pmod {d_i} \}$ (for $a \in \{0,\ldots,d_i-1\}$) are disjoint, there are at most $i$ progressions with modulus $d_i$ that cover at least a $1/i$ proportion of $R_i$. It follows that the number of choices for $\cA$ is at most $(n!)^2$, as claimed.
\end{proof}

It therefore only remains to bound the number of choices for the moduli. Note that if $\lcm(\cA) = p_1^{\g_1} \cdots p_m^{\g_m}$ then we have at most $\prod_i \big( \g_i + 1 \big)$ choices for each modulus. The following lemma, which we will use again later, provides a sharp bound on this product. }

\begin{lemma}\label{lem:numSubGroups} 
Let $(p_1, \ldots, p_m)$ be a sequence of distinct primes, let $(\g_1, \ldots, \gamma_m)$ be a sequence of positive integers, and let $M \ge \sum_{i = 1}^m \g_i (p_i -1)$. Then
\begin{equation}\label{equ:objfun} 
\sum_{i = 1}^m \log\big( \g_i + 1 \big) \le \big( 2\sqrt{\tau} + o(1) \big) \left( \frac{M}{\log M} \right)^{1/2}
\end{equation}
as $M \to \infty$. 
\end{lemma}

{\setstretch{1.07}
\begin{proof} 
We may assume that $p_1 < \cdots < p_m$, and reorder the $\gamma_i$ so that $\g_1 \geq \cdots \geq \g_m$, noting that this does not change the left-hand side of~\eqref{equ:objfun}, and that the inequality $M \ge \sum_{i = 1}^m \g_i (p_i -1)$ still holds under the new ordering. Set $x_t := \max\big\{ i : \g_i \ge t \big\}$, and observe that maximizing the left-hand side of~\eqref{equ:objfun} is equivalent to maximizing
$$X :=  \sum_{t \ge 1} x_t \cdot \log\left( \frac{t+1}{t} \right)$$ 
subject to the constraint
\begin{equation}\label{equ:constraint}
M \ge \sum_{i \ge 1} \g_i(p_i-1) = \sum_{t\geq 1} \sum_{i=1}^{x_t} (p_i-1) \ge \sum_{t \ge 1} \frac{x_t^2}{2} \max\big\{ \log x_t - 3, \, 1 \big\},
\end{equation}
where in the final step we used the following bound of Massias and Robin~\cite{MR},
$$\sum_{i=1}^x p_i \geq \frac{x^2}{2}\left( \log x + \log\log x - \frac{3}{2} - \frac{3.568}{\log x} \right) \geq \frac{x^2}{2} \big( \log x - 2 \big),$$
which holds for every $x \ge e^3$. Note that $X$ is increasing in $x_t$ for each $t \ge 1$, and so (by allowing $0 \le x_t \in \R$) we may assume that $M$ is equal to the right-hand side of~\eqref{equ:constraint}. 

\enlargethispage*{\baselineskip}

Applying the method of Lagrange multipliers, it follows that there exists $\lambda \in \R$ such that, for each $t \ge 1$, either
\begin{equation} \label{equ:lagrange} 
\l \log \frac{t+1}{t} = x_t \bigg( \log x_t - \frac{5}{2} \bigg),
\end{equation} 
or $x_t \le e^4$. We will first show that the contribution to $X$ of those values of $t$ such that $x_t = O(1)$ is small. To do so, note that $x_t = 0$ for all $t > M$, by~\eqref{equ:constraint}, and observe that therefore
\begin{equation}\label{equ:small:xts} 
\sum_{t \ge 1} x_t \cdot \log\left( \frac{t+1}{t} \right) \mathbbm{1}\big[ x_t \le \log M \big] \, \le \, \big( \log (M+1) \big)^2.
\end{equation} 
We may therefore restrict our attention to those values of $t$ for which $x_t > \log M > e^4$, so that, in particular,~\eqref{equ:lagrange} holds. Let $T = \max\big\{ t : x_t > \log M \big\}$ and observe that, by~\eqref{equ:constraint},~\eqref{equ:lagrange} and~\eqref{equ:small:xts}, we have
$$X \le \lambda \sum_{t = 1}^T \frac{\big( \log \frac{t+1}{t} \big)^2}{\log x_t - 5/2} + O\big( \log M \big)^2 \quad\qquad \text{and} \quad\qquad M \ge \frac{\lambda^2}{2} \sum_{t = 1}^T \frac{\left(\log\frac{t+1}{t} \right)^2}{\log x_t}.$$
To bound these sums, observe that $\lambda \to \infty$ as $M \to \infty$ 
(by~\eqref{equ:lagrange} and since $x_1 > \log M$), and that therefore, uniformly in $1 \le t \le \log \lambda$, we have
$$\log x_t = \log \l + \log\log \frac{t+1}{t} - \log\big( \log x_t - 5/2 \big) = \big( 1 + o(1) \big)\log \l$$
as $M \to \infty$. Moreover, if $\log \lambda \le t \le T$, then $\log\frac{t+1}{t} \le 1/t$ and $\log x_t \ge \log \log M$. 
It follows that, for each fixed $c \in \R$, we have  
$$\sum_{t = 1}^T \frac{\left(\log\frac{t+1}{t} \right)^2}{\log x_t - c} = \frac{1+o(1)}{\log \l} \sum_{t\geq 1} \left( \log \frac{t+1}{t} \right)^2 = \frac{\tau + o(1)}{\log \l},$$
as $M \to \infty$, and hence $M \ge \big( \tau/2 + o(1) \big) \lambda^2 / \log \l$. Finally, we deduce that
$$X \le \big( \tau + o(1) \big) \frac{\lambda}{\log \l} \le \big( 2\sqrt{\tau} + o(1) \big) \left( \frac{M}{\log M} \right)^{1/2}$$
as $M \to \infty$, as required.
\end{proof}
}
{\setstretch{1.2}

We can now easily deduce Proposition~\ref{lem:LargemCrudeBound}.

\begin{proof}[Proof of Proposition~\ref{lem:LargemCrudeBound}]
We first choose the moduli of the progressions in $\cA = \{A_1,\ldots,A_n\}$, and then the shifts. Since $\lcm(\cA) = N = p_1^{\g_1} \cdots p_m^{\g_m}$, for each $j \in [n]$ we have at most 
$$\prod_{i = 1}^m \big( \g_i + 1 \big) \le \exp\bigg( \big( 2\sqrt{\tau} + o(1) \big) \left( \frac{n}{C\log(n/C)} \right)^{1/2} \bigg)$$
choices for the modulus of the arithmetic progression $A_j$, where the inequality follows by applying Lemma~\ref{lem:numSubGroups} with $M = n / C$, and using our bound on $n$. By Lemma~\ref{lem:counting:shifts}, it follows that the number of choices for $\cA$ is at most
$$(n!)^2 \cdot \exp\bigg( \bigg( \frac{2\sqrt{\tau}}{\sqrt{C}} + o(1) \bigg) \frac{n^{3/2}}{\big( \log(n/C) \big)^{1/2}} \bigg) = \exp\left( \left( \frac{2\sqrt{\tau}}{\sqrt{C}} + o(1)\right)\frac{n^{3/2}}{(\log n)^{1/2}} \right)$$
as $n \to \infty$, as required.
\end{proof}

Using Simpson's theorem (Theorem~\ref{Simpson:thm}), we can now easily deduce an upper bound on the number of minimal covering systems that is sharp up to a constant factor in the exponent. 

\begin{corollary}
The number of minimal covering systems of\/ $\Z$ of size $n$ is 
$$\exp\bigg( \frac{\Theta\big( n^{3/2} \big)}{(\log n)^{1/2}  } \bigg).$$
\end{corollary}

\begin{proof}
The lower bound follows by Proposition~\ref{lem:NumbOfFrames} (or by the simpler construction in the introduction). For the upper bound, recall that, by Simpson's theorem, we have
$$|\cA| \ge \sum_{i=1}^m \gamma_i \big( p_i - 1 \big) + 1 \ge \sum_{i=1}^m \gamma_i \log_2 p_i$$
for any minimal covering system $\cA$ of $\Z$ with $\lcm(\cA) = N = p_1^{\g_1} \cdots p_m^{\g_m}$, and hence $N \le 2^n$. Thus, applying Proposition~\ref{lem:LargemCrudeBound} with $C = 1$ (and summing over $N \le 2^n$), there are at most 
$$\exp\left( \big( 2\sqrt{\tau} + o(1) \big) \frac{n^{3/2}}{(\log n)^{1/2}} \right)$$
minimal covering systems of $\Z$ of size $n$, as required.
\end{proof}
}

\section{Proof of Theorem~\ref{thm:MainCountingThm}}\label{MainCountingProofSec}

{\setstretch{1.12}

In this section we will complete the proof of Theorem~\ref{thm:MainCountingThm}; we begin by giving an overview of the remaining part of the argument. Let $\cA$ be a minimal covering system of $\Z$ of size $n$, let $N = \lcm(\cA)$ and, recalling~\eqref{def:bracketsN}, set $S_{(p,e)} = \{0,\ldots,p-1\}$ for each $(p,e) \in \< N \>$. We map $\Z_N$ into $S_{\< N \>} = \prod_{(p,e) \in \< N \>} S_{(p,e)}$ as described in Section~\ref{arithmetic:sec}; that is, 
we associate $x \in \Z_N$ with the vector $y = \varphi_N(x) \in S_{\< N \>}$, where $y_{(p,e)}$ is the coefficient of $p^{e-1}$ in the $p$-ary expansion of $x$ modulo $p^e$. Note that the image of each progression in $A \in \cA$ is a hyperplane in $S_{\< N \>}$.  Moreover, by Observation~\ref{obs:arithmetic:hyperplanes}, if $H = \varphi_N(A)$ then, for each prime $p$, the set 
$$\big\{ (p,e) \in \< N \> : (p,e) \in F(H) \big\}$$ 
forms a (possibly empty) initial segment of the sequence $(p,1), (p,2), \ldots \,$. Recall that we call hyperplanes that satisfy this condition `arithmetic'. 

We will apply Theorem~\ref{thm:mainStructureThm} to $\cA$ (with $C = 4$ and $\eps > 0$ an arbitrarily small constant), and deduce that either~\eqref{eq:structure:assumption} fails to hold, \emph{or} $\cA$ contains an almost optimal $\delta$-generalized frame $\big( \FF_{(p,e)} : (p,e) \in \< N \> \big)$. In the former case we are done by Proposition~\ref{lem:LargemCrudeBound}, so let us assume the latter. We will carefully count the number of choices for the fixed sets of the frame elements $\FF_{(p,e)}$ such that $p > \delta^{-1}$. The bound we obtain will be sufficiently strong unless $N$ is primarily composed of primes smaller than $\delta^{-1}$; however, for such $N$ it turns out that the simpler argument used in Section~\ref{counting:weird:frames:sec} suffices to give a sufficiently strong bound. 

Next, we bound the number of choices for the fixed sets of the remaining hyperplanes: those in frame sets $\FF_{(p,e)}$ for some prime $p \le \delta^{-1}$, and those not used in the frame. Surprisingly, it turns out that we can again obtain a sufficiently strong bound using the method of Section~\ref{counting:weird:frames:sec}. Roughly speaking, these `extra' hyperplanes are being used inefficiently, and would be better off (in terms of increasing the number of choices) by contributing to the construction of a larger frame (and thus a different value of $N$). 

Finally, noting that the fixed sets of the hyperplanes in $\cA$ correspond to the moduli of the original arithmetic progressions, we will use Lemma~\ref{lem:counting:shifts} to bound the number of minimal covering systems of $\Z$ of size $n$ with given moduli. 

\subsection{Choosing the fixed sets of $\delta$-generalized frames}\label{counting:frames:sec}

Let $N \in \N$ and $\delta > 0$, and suppose that $\big( \FF_{(p,e)} : (p,e) \in \< N \> \big)$ is a $\delta$-generalized frame in $S_{\< N \>}$ consisting of arithmetic hyperplanes.\footnote{By Definition~\ref{def:generalizedframe}, we may suppose that $\delta$ is sufficiently small; in particular, we will assume that $\delta < 1/2$.} Recall that $\FF_{(p,e)}$ is a collection of at most $p - 1$ hyperplanes, and that there exists an ordering $\prec$ on $\< N \>$, and for each $(p,e) \in \< N \>$ a set 
\begin{equation}\label{eq:Ipe:contains}
I(p,e) \supseteq \big\{ (q,f) \in \< N\> : (p,e) \prec (q,f) \text{ and } (q,f) \ne (p,e) \big\},
\end{equation}
such that $\mu_{I(p,e)}(H) > \delta$ for each $(p,e) \in \<N\>$ and $H \in \FF_{(p,e)}$. Recall also that $(p,e) \in F(H)$, and that the sets $\FF_{(p,e)}$ with $p > \delta^{-1}$ are disjoint. We remark that the ordering $\prec$ might not be arithmetic, but the hyperplanes \emph{are} arithmetic, and this will turn out to be sufficient.

}
{\setstretch{1.18}

In this subsection we will bound the number of choices for the fixed sets of the hyperplanes in $\big( \FF_{(p,e)} : (p,e) \in \< N \> \big)$ corresponding to primes larger than $\delta^{-1}$. While doing so, it will be convenient to write $\< N \>_\delta := \big\{ (p,e) \in \< N \> : p > \delta^{-1} \big\}$, and to define
$$\G(N) := \sum_{(p,e) \in \< N \>} (p-1) \qquad \text{and} \qquad \G_{\delta}(N) := \sum_{(p,e) \in \la N \ra_{\delta}} (p-1).$$
Note that, by Simpson's theorem, if $\lcm(\cA) = N$ then $|\cA| \ge \G(N)$. Given an ordering $\prec$ on $\< N \>$, for each $(p,e) \in \< N \>$ set
$$M_\prec(p,e) := \prod_{(q,f) \prec (p,e)} q,$$
and given a collection $\cA$ of hyperplanes, let us write $\cD(\cA) := \big( F(H) : H \in \cA \big)$ for the corresponding collection of fixed sets. We begin by observing the following upper bound (cf.~Lemma~\ref{lem:general:lower:bound}) on the number of choices for the sequence $\big( \cD(\FF_{(p,e)}) : (p,e) \in \< N \>_\delta \big)$.

\begin{lemma}\label{lem:general:upper:bound}
Let $N \in \N$ and $\delta > 0$, and let $\prec$ be an ordering on $\< N \>$. There are at most 
$$\exp\Bigg( \sum_{(p,e) \in \< N \>_\delta } (p-1) \Bigg( \sum_{(q,f) \in \< M_\prec(p,e) \>} \log\left( \frac{f+1}{f} \right) + \frac{1}{\delta^2} \Bigg) \Bigg)$$ 
sequences $\big( \cD(\FF_{(p,e)}) : (p,e) \in \< N \>_\delta \big)$ such that $\big( \FF_{(p,e)} : (p,e) \in \< N \> \big)$ is a simple $\delta$-generalized frame in $S_{\< N \>}$ with ordering~$\prec$ and consisting only of arithmetic hyperplanes.
\end{lemma}

\begin{proof}
Let $(p,e) \in \la N \ra_{\delta}$, let $H \in \FF_{(p,e)}$, and let $q$ be a prime. Recall that, since $H$ is an arithmetic hyperplane, it follows by Observation~\ref{obs:arithmetic:hyperplanes} that the set $F(H)$ induces a (possibly empty) initial segment of the set $(q,1), (q,2), (q,3), \ldots$

Suppose first that $q > \delta^{-1}$. We claim in this case that there are at most 
$$\big| \big\{ f : (q,f)  \in \< M_\prec(p,e) \> \big\} \big| + 1 = \exp\Bigg( \sum_{f \, :\, (q,f)  \in \< M_\prec(p,e) \>} \log \left( \frac{f+1}{f} \right) \Bigg)$$
choices for this initial segment. To see this, recall that $\mu_{I(p,e)}(H) > \delta$, and therefore $(q,f) \not\in F(H)$ for every $(q,f) \in I(p,e)$. By~\eqref{eq:Ipe:contains}, it follows that $F(H)$ does not contain any element $(q,f)$ with $(p,e) \prec (q,f)$ and $(p,e) \ne (q,f)$, and therefore the elements $(q,f)$ in $F(H)$ form an initial segment (in increasing order of~$f$) of the set $\big\{ (q,f) \in \< N \> :  (q,f) \prec (p,e) \big\}$. Since this set has the same size as the set $\big\{ f : (q,f)  \in \< M_\prec(p,e) \> \big\}$ (which is an initial segment of the positive integers), the claimed bound on the number of choices follows.

Now suppose instead that $q \le \delta^{-1}$. In this case the condition $\mu_{I(p,e)}(H) > \delta$ only implies that $F(H)$ contains at most $\log_2(\delta^{-1})$ elements of $I(p,e)$, and hence, by~\eqref{eq:Ipe:contains}, at most $\log_2(\delta^{-1})$ elements $(q,f)$ such that $(q,f) \not\prec (p,e)$. Repeating the argument from the case $q > \delta^{-1}$, it follows that we have at most 
$$\big| \big\{ f : (q,f)  \in \< M_\prec(p,e) \> \big\} \big| + 1 + \log_2(\delta^{-1})  \le \exp\Bigg( \sum_{f \, :\,(q,f) \in \< M_\prec(p,e) \>} \log \left( \frac{f+1}{f} \right)  + \frac{1}{\delta}  \Bigg)$$
choices for the initial segment of the set $(q,1), (q,2), (q,3), \ldots$ 

Finally, recall that there are at most $p - 1$ hyperplanes in $\FF_{(p,e)}$ for each $(p,e) \in \la N \ra_{\delta}$, and note that there are at most $\delta^{-1}$ primes $q \le \delta^{-1}$. Hence, multiplying the number of choices for all $(p,e) \in \la N \ra_{\delta}$, all $H \in \FF_{(p,e)}$, and all primes $q$ that divide~$N$, it follows that we have at most 
$$\exp\Bigg( \sum_{(p,e) \in \< N \>_\delta } (p-1) \Bigg( \sum_{(q,f) \in \< M_\prec(p,e) \>} \log\left( \frac{f+1}{f} \right) + \frac{1}{\delta^2} \Bigg) \Bigg)$$ 
choices for the sequence $\big( \cD(\FF_{(p,e)}) : (p,e) \in \< N \>_\delta \big)$, as claimed.
\end{proof}
}
{\setstretch{1.15}

For each $N \in \N$ and $\delta > 0$, and each ordering $\prec$ on $\< N \>$, let us define 
$$Q_\delta(N, \prec) := \sum_{(p,e) \in \< N \>_\delta } (p-1) \sum_{(q,f) \in \< M_\prec(p,e) \>} \log\left( \frac{f+1}{f} \right).$$
The following lemma provides a sufficiently strong upper bound on $Q_\delta(N, \prec)$. 

\begin{lemma} \label{lem:FrameCount}
Let $N \in \N$ and $\delta > 0$, and let $\prec$ be an ordering on $\< N \>$. If\/ $\G_{\delta}(N) > \delta \cdot \G(N)$, then 
\begin{equation}\label{eq:FrameCount}
Q_{\delta}(N,\prec) \leq \left( \frac{4\sqrt{\tau}}{3} + o(1)\right) \frac{\G_\delta(N)^{3/2}}{\big( \log \G_\delta(N) \big)^{1/2}}
\end{equation} 
as $N \to \infty$.
\end{lemma}

We first use Lemma~\ref{lem:numSubGroups} to obtain the following bound when $\G_\delta(M)$ is large. 

\begin{lemma} \label{cor:NumsubgroupsBigPrimes}
Let $M \in \N$ and $\delta > 0$. Then 
$$\sum_{(q,f) \in \la M \ra} \log\left(\frac{f+1}{f}\right) \leq \big( 2\sqrt{\tau} + o(1) \big) \left( \frac{\G_\delta(M)}{\log \G_\delta(M)} \right)^{1/2} + \frac{2}{\delta}\log \G(M)$$
as $\G_\delta(M) \to \infty$.
\end{lemma}

\begin{proof}
Let $M = p_1^{\g_1} \cdots p_m^{\g_m}$, and observe that
\begin{align*}
\sum_{(q,f) \in \la M \ra} \log\left(\frac{f+1}{f}\right) & \, = \sum_{i = 1}^m \log\big( \g_i + 1 \big) = \sum_{i \,:\, p_i > \delta^{-1}} \log\big( \g_i + 1 \big) + \sum_{i \,:\, p_i \le \delta^{-1}} \log\big( \g_i + 1 \big) \\
& \, \le \big( 2\sqrt{\tau} + o(1) \big) \left( \frac{\G_\delta(M)}{\log \G_\delta(M)} \right)^{1/2} + \frac{2}{\delta}\log \G(M)
\end{align*} 
as $\G_\delta(M) \to \infty$, as required, by Lemma~\ref{lem:numSubGroups} applied to the sequence $\g_i \cdot \mathbbm{1}\big[ p_i > \delta^{-1} \big]$, and by the bound $\g_i \le \G(M)$, which holds for every $i \in [m]$. 
\end{proof}

When $\G_\delta(M)$ is bounded and $\G(M) \to \infty$, we will instead use the bound
\begin{equation}\label{eq:GdeltaM:bounded}
\sum_{(q,f) \in \la M \ra} \log\left(\frac{f+1}{f}\right) \, \le \, \frac{3}{\delta}\log \G(M),
\end{equation}
which follows from the proof above by using the trivial bound $\log( \g_i + 1 ) \le \G_\delta(M)$ for the (bounded number of) large primes $p_i$, instead of applying Lemma~\ref{lem:numSubGroups}. 

We will also need the following easy lemma.

\begin{lemma}\label{lem:sumOfSquareRoots}
Let $2 \le m_0 < m_1 < \cdots < m_\ell \le m$. Then
$$\sum_{i = 0}^{\ell-1} \left(\frac{m_i}{\log m_i} \right)^{1/2} \big( m_{i+1} - m_i \big) \, \le \, \bigg( \frac{2}{3} + o(1) \bigg) \frac{m^{3/2}}{(\log m )^{1/2}}$$
as $m \to \infty$.
\end{lemma} 

We can now prove Lemma~\ref{lem:FrameCount}.

\begin{proof}[Proof of Lemma~\ref{lem:FrameCount}]
Observe first that, by Lemma~\ref{cor:NumsubgroupsBigPrimes} applied with $M = M_\prec(p,e)$, we have
$$\sum_{(q,f) \in \< M_\prec(p,e) \>} \log\left( \frac{f+1}{f} \right) \le \big( 2\sqrt{\tau} + o(1) \big) \bigg( \frac{\G_\delta\big( M_\prec(p,e) \big)}{\log \G_\delta\big( M_\prec(p,e) \big) } \bigg)^{1/2} + \frac{2}{\delta}\log \G\big( M_\prec(p,e) \big)$$
as $\G_\delta(M_\prec(p,e)) \to \infty$, and that when $\G_\delta(M_\prec(p,e))$ is bounded, by~\eqref{eq:GdeltaM:bounded} we have 
$$\sum_{(q,f) \in \< M_\prec(p,e) \>} \log\left(\frac{f+1}{f}\right) \, \le \, \frac{3}{\delta} \log \G(N),$$
as $\G(N) \to \infty$, since $\G(M_\prec(p,e)) \le \G(N)$. Thus, summing over $(p,e) \in \< N \>_\delta$, it follows that 
$$Q_\delta(N,\prec) \le \big( 2\sqrt{\tau} + o(1) \big) \sum_{(p,e) \in \la N \ra_\delta} (p-1)  \bigg( \frac{\G_\delta\big( M_\prec(p,e) \big)}{\log \G_\delta\big( M_\prec(p,e) \big)} \bigg)^{1/2} +  \frac{3}{\delta} \, \G(N) \log \G(N)$$ 
as $N \to \infty$ (which, in particular, implies that $\G(N) \to \infty$). 

Next, we apply Lemma~\ref{lem:sumOfSquareRoots} with $(m_0,\ldots,m_\ell) = \big(  \G_\delta\big( M_\prec(p,e) \big) \big)_{(p,e) \in \< N \>_\delta}$ and $m = \G_\delta(N)$. Note that if $(p',e')$ immediately follows $(p,e)$ in the ordering $\prec$ restricted to $\la N \ra_{\delta}$, then 
$$\G_\delta\big( M_\prec(p',e') \big) - \G_\delta\big( M_\prec(p,e) \big) = p - 1,$$
since $\G_\delta$ only counts the large primes. It follows that
$$Q_\delta(N,\prec) \, \le \left( \frac{4\sqrt{\tau}}{3} + o(1)\right) \frac{\G_\delta(N)^{3/2}}{\big( \log \G_\delta(N) \big)^{1/2}} + \frac{3}{\delta} \, \G(N) \log \G(N),$$
as $N \to \infty$. Since $\G_{\delta}(N) > \delta \cdot \G(N)$, by assumption, we obtain~\eqref{eq:FrameCount}, as required.
\end{proof}

Before continuing, let us observe that the condition $\G_{\delta}(N) > \delta \cdot \G(N)$ in the statement of Lemma~\ref{lem:FrameCount} (which in any case could be weakened considerably) is not a serious restriction, since we can easily obtain, using the method of Section~\ref{counting:weird:frames:sec}, a suitable bound on the number of minimal covering systems whose least common multiple has mostly small prime factors. 

}
{\setstretch{1.2}

\begin{lemma}\label{lem:many:small:primes} 
Let $\beta,\delta > 0$ be constants and let $n \in \N$ and $N \in \N$ with $\G_\delta(N) \le \beta \cdot \G(N)$. The number of minimal covering systems $\cA$ of\/ $\Z$ of size $n$ with $\lcm(\cA) = N$ is at most
$$\exp\bigg( \big( 2\sqrt{\beta \tau} + o(1) \big) \frac{n^{3/2}}{(\log n)^{1/2}} \bigg)$$
as $n \to \infty$.
\end{lemma}

\begin{proof}
The proof is essentially the same as that of Proposition~\ref{lem:LargemCrudeBound}, but we use Lemma~\ref{cor:NumsubgroupsBigPrimes} in place of Lemma~\ref{lem:numSubGroups} to count the choices of the moduli. To be more precise, in order to count the minimal covering systems $\cA = \{A_1,\ldots,A_n\}$ of $\Z$ of size $n$ with $\lcm(\cA) = N$, we will first choose the moduli, and then the shifts. Observe first that, for each $j \in [n]$, we have at most 
$$\prod_{(p,e) \in \la N \ra} \left(\frac{e+1}{e}\right) \le \exp\bigg( \big( 2\sqrt{\tau} + o(1) \big) \left( \frac{\G_\delta(N)}{\log \G_\delta(N)} \right)^{1/2} + \frac{2}{\delta}\log \G(N) \bigg)$$
choices for the modulus of the arithmetic progression $A_j$. Indeed, the left-hand side is simply the number of divisors of $N$, and the inequality follows from Lemma~\ref{cor:NumsubgroupsBigPrimes}. 

Now, since (by assumption and by Simpson's theorem) we have $\G_\delta(N) \le \beta \cdot \G(N) \le \beta n$, it follows by Lemma~\ref{lem:counting:shifts} that the number of choices for $\cA$ is at most
$$(n!)^2 \cdot \exp\bigg( \big( 2\sqrt{\beta \tau} + o(1) \big) \frac{n^{3/2}}{(\log \beta n)^{1/2}} + \frac{2n \log n}{\delta} \bigg) = \exp\left( \big( 2\sqrt{\beta \tau} + o(1) \big) \frac{n^{3/2}}{(\log n)^{1/2}} \right)$$
as $n \to \infty$, as claimed.
\end{proof}

}

\subsection{Proof of Theorem~\ref{thm:MainCountingThm}}\label{FinalProofSec}

{\setstretch{1.2}

We are finally ready to put together the pieces and deduce our main counting result. We will need the following easy bound. 

\begin{lemma}\label{lem:techLemma} 
For every $m \in \N$ and $x \ge 0$, we have 
$$\frac{(m+x)^{3/2}}{(\log(m+x))^{1/2} }  \, \ge \, \frac{ m^{3/2}}{(\log m )^{1/2}} + \bigg( \frac{3}{2} + o(1) \bigg) \left( \frac{m}{\log m} \right)^{1/2} x,$$
as $m \to \infty$. 
\end{lemma}

We can now deduce Theorem~\ref{thm:MainCountingThm} from Theorem~\ref{thm:mainStructureThm}, Propositions~\ref{lem:NumbOfFrames} and~\ref{lem:LargemCrudeBound}, Lemma~\ref{lem:counting:shifts}, Simpson's theorem, and the results of this section. 

\begin{proof}[Proof of Theorem~\ref{thm:MainCountingThm}] 
The lower bound follows immediately from Proposition~\ref{lem:NumbOfFrames}, so we will prove the upper bound. Observe first (cf.~Section~\ref{counting:weird:frames:sec}) that, by Simpson's theorem, if $\cA$ is a minimal covering system of $\Z$ of size $n$, then $\lcm(\cA) \le 2^n$. We may therefore fix $N \le 2^n$, and consider only covering systems $\cA$ such that $\lcm(\cA) = N$. We associate each progression $A \in \cA$ with an arithmetic hyperplane in $S_{\< N \>}$ using the bijection $\varphi_N$, as described above. 

Let $\eps > 0$ be an arbitrarily small constant, set $C = 4$, and let $\delta = \delta(C,\eps) > 0$ be the constant given by Theorem~\ref{thm:mainStructureThm}. Suppose first that either $n > 4\G(N)$ or $\G(N) > 4 \G_\delta(N)$. Then, by Proposition~\ref{lem:LargemCrudeBound} and Lemma~\ref{lem:many:small:primes}, there are at most 
$$\exp\left( \big( \sqrt{\tau} + o(1) \big) \frac{n^{3/2}}{(\log n)^{1/2}} \right)$$
minimal covering systems $\cA$ of $\Z$ of size $n$ with $\lcm(\cA) = N$, as required. By Simpson's theorem, let us therefore assume from now on that $\G(N) \le n \le 4\G(N) \le 2^4 \G_\delta(N)$.

{\setstretch{1.25}

By Theorem~\ref{thm:mainStructureThm} (and our choice of $\delta$), every minimal covering system $\cA$ of $\Z$ of size $n \le 4\G(N)$ with $\lcm(\cA) = N$ contains a $\delta$-generalized frame $\big( \FF_{(p,e)} : (p,e) \in \< N \> \big)$, with 
\begin{equation}\label{eq:frame:size:final:proof}
\sum_{(p,e) \in \< N \>} |\FF_{(p,e)}| \ge (1 - \eps) \G(N).
\end{equation}
Since $\G(N) \le 4 \G_\delta(N)$, it follows by Lemmas~\ref{lem:general:upper:bound} and~\ref{lem:FrameCount} that the number of sequences $\big( \cD(\FF_{(p,e)}) : (p,e) \in \< N \>_\delta \big)$ such that $\big( \FF_{(p,e)} : (p,e) \in \< N \> \big)$ is a $\delta$-generalized frame in $S_{\< N \>}$ consisting only of arithmetic hyperplanes is at most
\begin{equation}\label{eq:frame:choices:final:proof}
\G(N)! \cdot \exp\bigg( \left( \frac{4\sqrt{\tau}}{3} + o(1)\right) \frac{\G_\delta(N)^{3/2}}{\big( \log \G_\delta(N) \big)^{1/2}} + \frac{\G_\delta(N)}{\delta^2} \bigg),
\end{equation}
where the factor of $\G(N)!$ bounds (noting that $\G(N) \ge |\< N \>|$) the number of choices for the ordering $\prec$ on $\< N\>$ associated with the $\delta$-generalized frame. 

We next need to count the choices of the moduli for the remaining arithmetic progressions in $\cA$, that is, those corresponding to hyperplanes that are not included in $\FF_{(p,e)}$ for any $(p,e) \in \< N \>_\delta$. Recall (from Definition~\ref{def:generalizedframe}) that the sets $\FF_{(p,e)}$ with $(p,e) \in \< N \>_\delta$ are pairwise disjoint, so there are exactly 
$$x := n - \sum_{(p,e) \in \< N \>_\delta} |\FF_{(p,e)}|,$$
such arithmetic progressions in $\cA$. 
We bound the number of choices for the fixed sets of these remaining hyperplanes in $\cA$ using Lemma~\ref{cor:NumsubgroupsBigPrimes}, which implies that we have at most
$$\prod_{(p,e) \in \la N \ra} \left(\frac{e+1}{e}\right) \le \exp\bigg( \big( 2\sqrt{\tau} + o(1) \big) \left( \frac{\G_\delta(N)}{\log \G_\delta(N)} \right)^{1/2} + \frac{2}{\delta}\log \G(N) \bigg)$$
choices for each. Combining this bound with~\eqref{eq:frame:choices:final:proof}, and recalling that $\G(N) \le n \le 2^4 \G_\delta(N)$, it follows that we have at most
\begin{equation}\label{eq:moduli:choices:final:proof}
\exp\bigg( \left( \frac{4\sqrt{\tau}}{3} + o(1)\right) \frac{\G_\delta(N)^{3/2}}{\big( \log \G_\delta(N) \big)^{1/2}} +\big( 2\sqrt{\tau} + o(1) \big) \left( \frac{\G_\delta(N)}{\log \G_\delta(N)} \right)^{1/2} x \bigg)
\end{equation}
choices for the moduli of the arithmetic progressions in $\cA$, given $N$ and $x$. 

In order to bound~\eqref{eq:moduli:choices:final:proof}, we apply Lemma~\ref{lem:techLemma} with $m = \G_\delta(N)$. Since $0 \le x \le n$ and $2^{-4} n \le m \le n$, we obtain
\begin{equation}\label{eq:final:proof:techlemma}
\frac{\G_\delta(N)^{3/2}}{\big( \log \G_\delta(N) \big)^{1/2}} \,+\,  
\frac{3}{2} \left( \frac{\G_\delta(N)}{\log \G_\delta(N)} \right)^{1/2} x \, \le \, \frac{\big( \G_\delta(N) + x \big)^{3/2}}{\big( \log( \G_\delta(N) + x ) \big)^{1/2}} \,+\, \frac{o(n^{3/2})}{(\log n)^{1/2}}
\end{equation}
as $n \to \infty$. Now, since $|\FF_{(p,e)}| \le p - 1$ for each $(p,e) \in \< N \>$, it follows from~\eqref{eq:frame:size:final:proof} that
$$n - x = \sum_{(p,e) \in \< N \>_\delta} |\FF_{(p,e)}| \, \ge \, \G_\delta(N) - \eps \G(N) \, \ge \, \G_\delta(N) - \eps n,$$ 
so $\G_\delta(N) + x \le (1+\eps)n$, and hence
$$\frac{\big( \G_\delta(N) + x \big)^{3/2}}{\big( \log( \G_\delta(N) + x ) \big)^{1/2}} \, \le \, (1+\eps)^{3/2} \frac{n^{3/2}}{(\log n)^{1/2}}.$$
Combining this with~\eqref{eq:moduli:choices:final:proof} and~\eqref{eq:final:proof:techlemma}, it follows that we have at most 
$$\exp\bigg( \left( \frac{4\sqrt{\tau}}{3} + o(1)\right) (1+\eps)^{3/2} \frac{n^{3/2}}{(\log n)^{1/2}} \bigg)$$
choices for the moduli of the progressions in $\cA$. 

Finally, by Lemma~\ref{lem:counting:shifts}, it follows that there are at most
$$\exp\bigg( \left( \frac{4\sqrt{\tau}}{3} + O(\eps) \right) \frac{n^{3/2}}{(\log n)^{1/2}} \bigg)$$
minimal covering systems of $\Z$ of size $n$ with $\lcm(\cA) = N$. Since $\eps > 0$ was arbitrarily small, this completes the proof of Theorem~\ref{thm:MainCountingThm}.}
\end{proof}

}
{\setstretch{1.15}
\section*{Acknowledgements}

This research was partly carried out during a one-month visit by the authors to IMT Lucca, and partly during visits by various subsets of the authors to IMPA and to the University of Memphis. We are grateful to each of these institutions for their hospitality, and for providing a wonderful working environment. We are also grateful to Christian Elsholtz for pointing out the appearance of the constant $\tau$ in the iterated divisor function.

}

\appendix

\section{Proof of the geometric Simpson's theorem}\label{sec:appendix:Simpson}

{\setstretch{1.15}

In this appendix we will provide, for the reader's convenience, a proof of the following slight generalization of Simpson's theorem~\cite{Simp}. 

\begin{theorem}[Simpson's theorem] \label{geometric:Simpson}
Let $\cA$ be a minimal cover of $S_{[k]}$ with hyperplanes, and let $I \subsetneq F(\cA)$. Then 
$$\big| \big\{ H \in \cA : F(H) \not \subseteq I \big\} \big| \, \ge \sum_{i \in F(\cA) \setminus I} \big( |S_i| - 1 \big) + 1.$$
\end{theorem}

Note that Theorem~\ref{Simpson:thm} follows from Theorem~\ref{geometric:Simpson} by setting $I = \emptyset$. 

\begin{proof}
The proof is by induction on $|F(\cA)|$. Set
$$L_s^{(i)} := S_1 \times \cdots S_{i-1} \times \{s\} \times S_{i+1} \times \cdots S_k$$ 
for each $i \in [k]$ and $s \in S_i$, and note that if $F(\AA) = \{i\}$, then $\cA$ (being minimal) must consist precisely of the $|S_i|$ hyperplanes $L_s^{(i)}$, one for each $s \in S_i$. Moreover, if $I \subsetneq F(\cA)$ then $I = \emptyset$, and hence
$$\big| \big\{ H \in \cA : F(H) \not \subseteq I \big\} \big| \, = \, |\AA| \, = \, |S_i|,$$
as required. So suppose that $|F(\cA)| \ge 2$, and (recalling that $I\subsetneq F(\AA)$) choose an element $i \in F(\cA) \setminus I$. For each $s \in S_i$, let $\cA_s \subseteq \cA$ be a minimal cover of $L_s^{(i)}$, and observe that 
$$\cA = \bigcup_{s \in S_i} \cA_s,$$ 
since $\cA$ is minimal. For convenience, let us assume (without loss) that $S_i = \{1,\ldots,p\}$. 

Now, set $F_i(H) := F(H) \setminus \{i\}$ for each $H \in \cA_s$ (and similarly for a family of hyperplanes), and define a sequence of sets $(R_0,\ldots,R_p)$ by setting $R_0 := I$, and 
$$R_s := R_{s-1} \cup F_i(\AA_s)$$ 
for each $s \in S_i$, so in particular $R_p = F(\cA) \setminus \{i\}$. Now set $I_s := R_{s-1} \cap F_i(\AA_s)$, and define
$$\cQ_s := \big\{ H \in \cA_s : F_i(H) \not \subseteq I_s \big\}.$$
We claim that, applying the induction hypothesis to the minimal cover $\cA_s$ of $L_s^{(i)}$ (which we naturally identify with $S_1 \times \cdots S_{i-1} \times S_{i+1} \times \cdots S_k$), we have either $R_{s-1}=R_s$, or 
\begin{equation}\label{app:Q:bound}
 |\cQ_s| \, \ge 
 \sum_{j \in R_s \setminus R_{s-1}} \big( |S_j| - 1 \big) + 1
\end{equation}
for each $s \in S_i$. To see this, simply note that $R_s \setminus R_{s-1} = F_i(\cA_s) \setminus I_s$, so if $R_{s-1} \ne R_s$ then $I_s \subsetneq F_i(\cA_s)$, and that $F_i(\cA_s) \subseteq F(\cA) \setminus \{i\}$, so (since $i \in F(\cA)$) we have $|F_i(\cA_s)| < |F(\cA)|$.

Set $J := \big\{ s \in S_i : R_{s-1} \ne R_s \big\}$, and recall that~\eqref{app:Q:bound} holds for each $s \in J$. We claim that
$$\bigg| \bigcup_{s \in J} \cQ_s \bigg| \, = \, \sum_{s \in J} |\cQ_s| \, \ge \, |J| - \big( |S_i| - 1 \big) + \sum_{j \in F(\cA)\setminus I} \big( |S_j| - 1 \big).$$
The inequality follows from summing~\eqref{app:Q:bound} over $s \in J$, and recalling that $i \in F(\cA)\setminus I$, so it remains to show that the sets $\cQ_s$ are disjoint. To see this, observe that, if $H \in \cA_s$, then 
$$F_i(H) \not \subseteq I_s \qquad \Leftrightarrow \qquad F_i(H) \subseteq R_s \quad \text{and} \quad F_i(H) \not\subseteq R_{s-1},$$ 
and so $H \in A_s$ for at most one element $s \in S_i$.

Finally, we claim that for each $s \in S_i \setminus J$, there exists a hyperplane $H \in \cA$ such that $H \subseteq L_s^{(i)}$ and $F(H) \not\subseteq I$. To see this, observe first that $L_s^{(i)}$ is not covered by $\big\{ H \in \cA : i \not\in F(H) \big\}$, as otherwise $S_{[k]}$ would be covered by $\big\{ H \in \cA : i \not\in F(H) \big\}$, contradicting the minimality of $\cA$ and the fact that $i \in F(\cA)$. It follows that there exists $H \in \cA$ with $H \subseteq L_s^{(i)}$, and we have $F(H) \not\subseteq I$ because $i \in F(H) \setminus I$. Moreover, none of these $|S_i| - |J|$ distinct hyperplanes is included in $\cQ_s$ for any $s \in J$, since they do not intersect the set $\bigcup_{s \in J} L_s^{(i)}$.

Hence, noting that $F(H) \not\subseteq I$ for each $H \in \cQ_s$ (since $I \subseteq R_{s-1}$ and $F(H) \not\subseteq R_{s-1}$), we obtain
$$\big| \big\{ H \in \cA : F(H) \not \subseteq I \big\} \big| \, \ge \, \bigg| \bigcup_{s \in J} \cQ_s \bigg| + |S_i| - |J| \, \ge \sum_{j \in F(\cA)\setminus I} \big( |S_j| - 1 \big) + 1,$$
as required.
\end{proof}
}
\end{document}